\documentclass[12pt]{article}
\usepackage{tikz}
\usetikzlibrary{arrows}
\usepackage[framemethod=tikz]{mdframed}
\usepackage{wrapfig}
\usepackage{amsmath,amssymb}
\usepackage{type1cm}
\usepackage{amsthm}
\usepackage{mathrsfs}
\usepackage{mathtools}
\usepackage{enumerate}
\usepackage[all]{xy} 
\usepackage[vcentermath,enableskew]{youngtab}
\usepackage{ytableau}
\usepackage{diagbox}
\AtBeginDocument{%
   \def\MR#1{}
}

\DeclareMathOperator{\GL}{GL}

\DeclareMathOperator{\GSp}{GSp}
\DeclareMathOperator{\Sp}{Sp}
\DeclareMathOperator{\ch}{ch}
\DeclareMathOperator{\diag}{diag}

\DeclareMathOperator{\Ad}{Ad}
\DeclareMathOperator{\Adm}{Adm}

\DeclareMathOperator{\Irr}{Irr}

\DeclareMathOperator{\pfn}{pfn}

\DeclareMathOperator{\supp}{supp}

\DeclareMathOperator{\LP}{LP}

\newcommand\F{\mathbb{F}}

\newcommand\Fq{\mathbb F_q}
\newcommand\aFq{\overline{\mathbb F}_q}
\newcommand\cA{\mathcal A}

\newcommand\cO{\mathcal O}

\newcommand\cG{\mathcal G}

\newcommand\cS{\mathcal S}

\newcommand\Gm{\mathbb G_m}

\newcommand\Q{\mathbb Q}
\newcommand\bQp{\breve{\Q}_p}
\newcommand\R{\mathbb R}
\newcommand\A{\mathbb A}

\newcommand\G{\mathbb G}
\newcommand\J{\mathbb J}

\newcommand\Z{\mathbb Z}
\newcommand\tW{\tilde W}

\newcommand\bS{\mathbb S}
\newcommand\SW{{^{\bS}\tilde W}}
\newcommand\SAdm{{^{\bS}\mathrm{Adm}}}

\newcommand\tS{\tilde{\mathbb S}}
\newcommand\inv{\operatorname{inv}}

\newcommand\ld{\lambda}

\newcommand\vph{\varphi}
\newcommand\vp{\varpi}
\newcommand\Y{X_*(T)}

\newcommand\la{\langle}
\newcommand\ra{\rangle}
\newcommand\pc{\preceq}
\newcommand\Xl{X^\ld_\mu(\tau)}

\newcommand\cAm{\cA_{\mu,\tau}}
\newcommand\cSm{\cS_{\mu,\tau}}
\newcommand\tP{\tilde \Phi}
\newcommand\ta{\tilde \alpha}

\theoremstyle{definition}
\newtheorem{theo}{Theorem}[section]
\newtheorem{prop}[theo]{Proposition}

\newtheorem{lemm}[theo]{Lemma}
\newtheorem{coro}[theo]{Corollary}
\newtheorem{exam}[theo]{Example}
\newtheorem{rema}[theo]{Remark}

\newtheorem{thm}{Theorem}[section]

\begin{document}
\title{Some generalizations of Oort's conjecture}
\author{Ryosuke Shimada and Teppei Takamatsu}
\date{}
\maketitle

\begin{abstract}
For a prime $p\geq5$, let $\mathscr S_g$ be the moduli space over $\overline{\F}_p$ of $g$-dimensional principally polarized supersingular abelian varieties.
We show that each of the following loci contains an open dense subscheme on which the principally polarized abelian varieties have automorphism group $\{\pm1\}$: (i) certain supersingular Ekedahl--Oort strata when $g$ is even, (ii) the loci in $\mathscr S_g$ with non-supersingular Ekedahl--Oort invariants of positive Coxeter type when $g\geq3$ and (iii) the locus in $\mathscr S_g$ with $a$-number at least $2$ when $g\geq4$.
Consequently, for $g\ge 4$, the complement in $\mathscr S_g$ of the open locus where the automorphism group is $\{\pm1\}$ has codimension at least $2$.
These results confirm Oort's conjecture for $p\ge 5$.
We reduce them to statements about affine Deligne--Lusztig varieties for $\GSp_{2g}$ and prove analogues of (ii) for $\GL_{2g}$ and $\operatorname{GSO}_{4m}$.
\end{abstract}

\section{Introduction}
\label{introduction}

Fix a prime $p$ and an integer $g\geq 1$, and let $\mathscr A_g$ be the moduli space over $\overline{\F}_p$ of principally polarized abelian varieties of dimension $g$ with a level $N$ structure for some integer $N$ coprime to $p$.
For an algebraically closed field $k$ of characteristic $p$ and $x\in \mathscr A_g(k)$, let $(A_x,\vph_x)$ denote the corresponding principally polarized abelian variety.
The moduli space $\mathscr A_g$ admits several natural stratifications defined by invariants of the associated abelian varieties.
The Newton and Ekedahl--Oort stratifications are defined by the isogeny class of $A_x[p^\infty]$ and the isomorphism class of $A_x[p]$, respectively.
The strata of the $a$-number stratification consist of points $x$ with the same value of $a(A_x)\coloneqq\dim_k\operatorname{Hom}_k(\alpha_p,A_x)$.
Let $\mathscr S_g$ denote the unique closed Newton stratum of $\mathscr A_g$, called the supersingular locus. Then $x$ lies in $\mathscr S_g$ if and only if $A_x$ is isogenous to a product of $g$ copies of a supersingular elliptic curve.

Chai and Oort \cite{CO11} proved that generically on each non-supersingular Newton stratum of $\mathscr A_g$, one has $\operatorname{Aut}(A_x,\vph_x)=\{\pm1\}$.
Oort \cite[Problem 4]{EMO01} also conjectured that, for $g\ge 2$, there is an open dense subscheme $U\subseteq\mathscr S_g$ such that $\operatorname{Aut}(A_x,\vph_x)=\{\pm1\}$ for every $x\in U(\overline{\F}_p)$, even though the endomorphism ring of every supersingular abelian variety of dimension $g$ has $\mathbb Z$-rank $4g^2$.

Ibukiyama \cite{Ibukiyama20} first proved the conjecture for $g=2$ and $p>2$ and also showed that it fails for $(g,p)=(2,2)$. Karemaker and Pries \cite{KP19} independently proved the conjecture for $g=2$ and $p>2$. Karemaker, Yobuko and Yu \cite{KYY21} proved it for $g=3$ and $p>2$ and showed that it fails for $(g,p)=(3,2)$. Dragutinovi\'c \cite{Dragutinovic24} proved the conjecture for $g=4$ and $p>2$, and also gave a new proof for $g=3$ and $p>2$.

A breakthrough came with the work of Karemaker and Yu \cite{KY24}, who proved that when $g$ is even and $p\geq5$, the automorphism group is $\{\pm1\}$ on an open dense subscheme of the maximal supersingular Ekedahl--Oort stratum. Here a supersingular Ekedahl--Oort stratum is an Ekedahl--Oort stratum contained in $\mathscr S_g$. This result implies Oort's conjecture for even $g$ when $p\geq5$. They also proved Oort's conjecture for $g=4$ and all $p$. Viehmann \cite{Viehmann26} finally proved Oort's conjecture for $g\geq2$ and every prime $p$ with $(g,p)\neq(2,2),(3,2)$, using Rapoport--Zink's uniformization theorem to reduce the problem to an explicit description of the $a=1$ locus of certain Rapoport--Zink moduli spaces of polarized supersingular $p$-divisible groups.

The perfections of reduced special fibers of Rapoport--Zink spaces are affine Deligne--Lusztig varieties. Let $\bQp$ be the completion of the maximal unramified extension of $\Q_p$, let $\breve{\Z}_p$ be its ring of integers, and let $\sigma$ be the Frobenius automorphism of $\bQp/\Q_p$. Set $K=\GSp_{2g}(\breve{\Z}_p)$, and let $\mu$ be the cocharacter of $\GSp_{2g}$ given by $z\mapsto\operatorname{diag}(z,\ldots,z,1,\ldots,1)$, where both $z$ and $1$ are repeated $g$ times. Let $p^\mu$ denote the image of $p$ under $\mu$. Set $\tau=\begin{pmatrix}0&p\cdot1_g\\1_g&0\end{pmatrix}$. The perfection of $\mathscr S_g$ is uniformized by the affine Deligne--Lusztig variety whose set of closed points is
$$X_\mu(\tau)(\overline{\F}_p)=\{xK\in\GSp_{2g}(\bQp)/K\mid x^{-1}\tau\sigma(x)\in Kp^\mu K\}.$$
The $\sigma$-centralizer $\J$ of $\tau$ acts on $X_\mu(\tau)$ by left multiplication.

Let $w$ be an element of the Iwahori--Weyl group $\tW=N_{\GSp_{2g}}(T)(\bQp)/T(\breve{\Z}_p)$, where $T\subseteq\GSp_{2g}$ is the diagonal torus.
At Iwahori level, one similarly has the affine Deligne--Lusztig variety $X_w(\tau)$, and we denote its image in the affine Grassmannian $\GSp_{2g}(\bQp)/K$ by $\pi(X_w(\tau))$.
Let $\Adm(\mu)\subset\tW$ be the $\mu$-admissible set in the Iwahori--Weyl group, and let $\SAdm(\mu)\subset\Adm(\mu)$ be the subset parametrizing the Ekedahl--Oort strata (cf.\ \S\ref{ADLV}).
If $w\in\SAdm(\mu)$, the intersection of $\mathscr S_g$ with the Ekedahl--Oort stratum attached to $w$ is uniformized by $\pi(X_w(\tau))$ (cf.\ \cite[\S5.1]{Wang21}).

By \cite[Proposition 5.6]{GHN19}, the Ekedahl--Oort stratum attached to $w$ is supersingular if and only if $\supp_\sigma(w)\coloneqq\bigcup_{n\in\mathbb Z}\tau^n(\supp(w\tau^{-1}))$ is a proper subset of the set $\tS$ of simple affine reflections.
If this is the case, then $\J$ acts transitively on the set of irreducible components of $\pi(X_w(\tau))$, which are pairwise disjoint.
Moreover, each irreducible component is isomorphic to a Deligne--Lusztig variety (see also \cite{Harashita10}), and its stabilizer in $\J$ is a parahoric subgroup.
In general, when $\supp_\sigma(w)=\tS$, the geometry of $\pi(X_w(\tau))$ is very complicated.
However, the first author, together with Schremmer and Yu, proved in \cite{SSY23} that the same conclusions hold for $\pi(X_w(\tau))$ when $w$ is of positive Coxeter type, except that each irreducible component is a product of a Deligne--Lusztig variety and an affine space.
Roughly speaking, an element $w$ is of positive Coxeter type if its finite part is a Coxeter element (cf.\ \S\ref{DL method}).

Recall that the $a$-number of an abelian variety can also be computed from the corresponding Dieudonn\'e module. Since each point $xK\in X_\mu(\tau)$ corresponds to the Dieudonn\'e lattice $x\breve{\Z}_p^{2g}$, this defines the $a=1$ and $a\geq2$ loci in $X_\mu(\tau)$.

By studying these strata in $X_\mu(\tau)$ and the action of $\J$, we prove the following.
\begin{thm}
\label{main theorem introduction}
Assume that $p\geq5$. Each of the following loci contains an open dense subscheme $U$ such that $\operatorname{Aut}(A_x,\vph_x)=\{\pm1\}$ for every $x\in U(\overline{\F}_p)$:
\begin{enumerate}[(i)]
\item for even $g$, the supersingular Ekedahl--Oort stratum attached to $w\in\SAdm(\mu)$ such that $\supp_\sigma(w)$ is a maximal proper $\tau$-stable subset of $\tS$;
\item for $g\geq3$, the intersection of $\mathscr S_g$ with the Ekedahl--Oort stratum attached to $w\in\SAdm(\mu)$ such that $\supp_\sigma(w)=\tS$ and $w$ is of positive Coxeter type;
\item for $g\geq4$, the locus in $\mathscr S_g$ where the $a$-number is at least $2$.
\end{enumerate}
\end{thm}

For (i), (ii) and (iii), see Theorem \ref{maximal supersingular EO}, Theorem \ref{main theo} and Theorem \ref{main theorem a-number}, respectively. The maximal supersingular Ekedahl--Oort stratum satisfies the assumptions of (i) (cf.\ Example \ref{maximal supersingular EO example}), so (i) generalizes \cite[Theorem A]{KY24}. In their notation, the union of the strata in (i) is $\mathscr S_{g,\frac{g}{2}}^{\mathrm{eo}}$. The $a=1$ locus in $\mathscr S_g$ is the intersection of $\mathscr S_g$ with an Ekedahl--Oort stratum satisfying the assumptions of (ii) (cf.\ Remark \ref{dense Coxeter stratum}), so (ii) generalizes \cite[Theorem 1.1]{Viehmann26} for $p\ge 5$. We further prove analogues of (ii) for $\GL_{2g}$ and $\operatorname{GSO}_{4m}$ (cf.\ Theorem \ref{main theo GL} and Theorem \ref{main theo GSO}). Viehmann also proved such an analogue for the $a=1$ locus for $\GL_{2g}$ \cite[Theorem 5.1]{Viehmann26}. The assertion (iii) is a much stronger statement than Oort's conjecture (cf.\ Remark \ref{Oort conjecture from a-number loci}).

Let $V\subset\mathscr A_g$ be the open subscheme characterized by
$$V(\overline{\F}_p)=\{x\in\mathscr A_g(\overline{\F}_p)\mid \operatorname{Aut}(A_x,\vph_x)=\{\pm1\}\}.$$
For $g\geq 2$ and $p\geq 5$, Oort's conjecture and the description of the geometric fibers show that $\mathscr S_g\setminus V$ is the non-\'{e}tale locus of the level-forgetting morphism to the coarse moduli space of supersingular principally polarized abelian varieties.
Combining a slight strengthening of Theorem \ref{main theorem introduction} (ii) with Theorem \ref{main theorem introduction} (iii), we obtain the following.
\begin{thm}
Assume that $g\geq4$ and $p\geq5$. Then
$\operatorname{codim}_{\mathscr S_g}(\mathscr S_g\setminus V)\geq2.$
\end{thm}
See Corollary \ref{codimension automorphism locus} for this theorem. For $g\geq3$, Oort \cite[Theorem (2.AV)]{Oort77} proved that $\operatorname{codim}_{\mathscr A_g}(\mathscr A_g\setminus V)\geq2.$ In work in preparation \cite{STinprep}, by extending the techniques developed in the present paper, we determine $\operatorname{codim}_{\mathscr S_g}(\mathscr S_g\setminus V)$ and $\operatorname{codim}_{\mathscr A_g}(\mathscr A_g\setminus V)$. 

In \S2.5, we explain how these problems are reduced to problems on $X_\mu(\tau)$.
After this reduction, for each locus in Theorem \ref{main theorem introduction}, we study an irreducible component $C$ whose stabilizer is $\J\cap P$, where $P$ is a standard parahoric subgroup of $\GSp_{2g}(\bQp)$.
The irreducible components form one $\J$-orbit in (i) and (ii), whereas they fall into $\left\lfloor\frac{g}{2}\right\rfloor$ $\J$-orbits in (iii), making the proof of (iii) more involved. We need to show that if $j\in\J\cap P$ has finite order and fixes $C$ pointwise, then $j=\pm1$. For this, we consider a $\J\cap P$-equivariant projection to a partial flag variety of the reductive quotient of $P$. This projection is constructed by Deligne--Lusztig reduction in (ii) and by the $\J$-stratification in (iii). We then apply Proposition \ref{center} concerning the action of $\J\cap P$ on Deligne--Lusztig varieties to show that $j$ is central modulo a torsion-free subgroup. Finally, an explicit calculation using root subgroups shows that $j=\pm1$.

These results suggest several interesting questions, including the classification of the Ekedahl--Oort strata satisfying the conclusions of Theorem \ref{main theorem introduction} (i) or (ii), the smallest dimension among such strata, the largest $a$-number for which the conclusion of Theorem \ref{main theorem introduction} (iii) holds. 
The methods developed in this paper will be useful in addressing these questions, including when $p=2$ or $3$. The $\J$-stratification, introduced by Chen and Viehmann \cite{CV18}, should be particularly important.
Indeed, at least for $6\le g\le 16$, only three Ekedahl--Oort strata of $a$-number $2$ have intersections with $\mathscr S_g$ of codimension $1$ in $\mathscr S_g$, whereas the $\left\lfloor\frac{g}{2}\right\rfloor$ orbits of irreducible components of the $a\geq2$ locus are represented by closures of $\J$-strata.
Since there is no known canonical way to decompose an Ekedahl--Oort stratum, the $\J$-stratification is essential to our proof of Theorem \ref{main theorem introduction} (iii).
The $\J$-stratification also appears in \cite[Remark 6.21]{KY24}.

As we finalized this paper, Karemaker and Yu \cite{KY26} posted an independent proof of Oort's conjecture for odd $p$. They construct, for every $g\ge 3$, a $(g-1)$-dimensional closed subvariety of $\mathscr S_g$ with an open dense subvariety of $a$-number $g-2$ on which the automorphism group is $\{\pm1\}$. Their result provides further motivation for these questions.

The paper is organized as follows.
In \S\ref{preliminaries}, we recall some basic facts about affine Deligne--Lusztig varieties, the Deligne--Lusztig reduction method and the $\J$-stratification. 
In \S3, we study Ekedahl--Oort strata of positive Coxeter type and prove Theorem \ref{main theorem introduction} (i) and (ii), together with analogues of (ii) for $\GL_{2g}$ and, when $g$ is even, $\operatorname{GSO}_{2g}$. In \S4, we describe the irreducible components of the $a\geq2$ locus using parahoric $\J$-strata and prove Theorem \ref{main theorem introduction} (iii) and Theorem B.

\textbf{Acknowledgments:}
The first author became interested in this problem thanks to Eva Viehmann.
The authors thank Eva Viehmann and Chia-Fu Yu for helpful comments.

The first author was supported by JSPS KAKENHI Grant number JP25K23334.
The second author was supported by JSPS KAKENHI Grant number JP25K17228.
The first author worked on this paper during a stay at UC Berkeley supported by the JSPS
Overseas Research Fellowship. The first author would like to thank the university and
his host, Sug Woo Shin, for their hospitality.

\section{Preliminaries and reformulation}
\label{preliminaries}
Keep the notation in \S\ref{introduction}.
We sometimes drop the adjective ``perfect'' for notational convenience, although we need to work with perfect schemes in most statements.

\subsection{Notation}
\label{notation}
Let $F$ be a non-archimedean local field with residue field $\F_q$ of prime characteristic $p$, and let $L$ be the completion of the maximal unramified extension of $F$.
Let $\sigma$ denote the Frobenius automorphism of $L/F$ and $\aFq/\Fq$.
We write $\cO$ (resp.\ $\cO_F$) for the valuation ring of $L$ (resp.\ $F$).
We denote by $\vp$ a uniformizer of $F$ (and $L$).

Let $G$ be a split connected reductive group over $F$ and let $T$ be a split maximal torus of it.
We write $X^* (T)$ and $X_* (T)$ for the character and cocharacter groups of $T$, respectively.
For simplicity, we assume that the Dynkin diagram of $G$ is connected.
Let $B$ be a Borel subgroup of $G$ containing $T$. 
Let $\Phi=\Phi(G,T)$ denote the set of roots of $T$ in $G$.
We denote by $\Phi_+$ (resp.\ $\Phi_-$) the set of positive (resp.\ negative) roots distinguished by $B$.
Let $\Delta$ be the set of simple roots, and let $\rho \in X^* (T)_{\Q}$ denote the half-sum of the positive roots.
Let $X_*(T)_+$ be the set of dominant cocharacters.
For $\mu,\mu'\in X_*(T)$ (resp.\ $X_*(T)_{\Q}$), we write $\mu'\pc \mu$ if $\mu-\mu'$ is a non-negative integral (resp.\ rational) linear combination of positive coroots.
For a cocharacter $\mu\in X_*(T)$, let $\vp^{\mu}$ be the image of $\vp\in \mathbb G_m(F)$ under the homomorphism $\mu\colon\mathbb G_m\rightarrow T$.

The Iwahori--Weyl group $\tW$ is defined as the quotient $N_{G(L)}T(L)/T(\cO)$.
This can be identified with the semi-direct product $W_0\ltimes X_{*}(T)$, where $W_0$ is the finite Weyl group of $G$.
We denote the projection $\tW\rightarrow W_0$ by $p$.
Let $\bS\subset W_0$ denote the subset of simple reflections, and let $\tS\subset \tW$ denote the subset of simple affine reflections.
We often identify $\Delta$ and $\bS$.
The affine Weyl group $W_a$ is the subgroup of $\tW$ generated by $\tS$.
Then we can write the Iwahori--Weyl group as a semi-direct product $\tW=W_a\rtimes \Omega$, where $\Omega\subset \tW$ is the subgroup of length $0$ elements.
Moreover, $(W_a, \tS)$ is a Coxeter system.
We denote by $\le$ the Bruhat order on $\tW$ (see \cite[Subsection 1.8]{KR00} for example).
For any $J\subseteq \tS$, let $^J\tW$ be the set of minimal length representatives for the cosets in $W_J\backslash \tW$, where $W_J$ denotes the subgroup of $\tW$ generated by $J$.
We also have a length function $\ell\colon \tW\rightarrow \Z_{\geq 0}$ given as
$$\ell(w_0\vp^{\lambda})=\sum_{\alpha\in \Phi_+, w_0\alpha\in \Phi_-}|\langle \alpha, \lambda\rangle+1|+\sum_{\alpha\in \Phi_+, w_0\alpha\in \Phi_+}|\langle \alpha, \lambda\rangle|,$$
where $w_0\in W_0$ and $\lambda\in \Y$.
We may and do embed $\tW$ into the group of affine transformations of $\Y_{\R}$ so that the action of $w=\vp^\ld u$ is given by $v\mapsto uv+\ld$.
For $\alpha\in \Phi$, we denote by $s_\alpha$ the reflection which sends $\ld\in \Y$ to $\ld-\langle \alpha, \ld\rangle\alpha^\vee$, where $\alpha^\vee$ is the corresponding coroot of $\alpha$.
Then $\bS=\{s_\alpha\mid \text{$\alpha$ is a simple root}\}$.

Let $\tP=\Phi\times \Z$ be the set of affine roots.
We view $a=(\alpha,k)\in \tP$ as an affine function such that $a(v)=-\la \alpha,v\ra+k$ for $v\in \Y_\R$.
Let $s_a=\vp^{k\alpha^\vee}s_\alpha$ denote the corresponding affine reflection.
If $w=\vp^\ld u\in \tW$, then $w(\alpha,k)=(u\alpha,k+\la u\alpha,\ld\ra)$.
Set $\tP_+=\{(\alpha,k)\in \tP\mid k\geq 1\}\sqcup \{(\alpha,0)\in \tP\mid \alpha\in \Phi_-\}$.
Then $\tP=\tP_+\sqcup \tP_-$ with $\tP_-=-\tP_+$.
For $\alpha\in \Phi$, we define $\ta=(\alpha,0)\in \tP$ if $\alpha\in \Phi_-$ and $\ta=(\alpha,1)\in \tP$ if $\alpha\in \Phi_+$.
Let $\Pi$ denote the set of minus simple roots and the highest positive root of $\Phi$.
Then $\ell(s_a)=1$ if and only if $a=\pm\ta$ for $\alpha\in \Pi$.
In particular, $\tS=\{s_{\ta}\mid \alpha\in \Pi\}$, and we sometimes identify $\Pi$ with $\tS$.
See \cite[\S1.2]{Nie22} for this description.

For $w\in W_a$, we denote by $\supp(w)\subseteq \tS$ the set of simple affine reflections occurring in every (equivalently, some) reduced expression of $w$.
Note that $\tau\in \Omega$ acts on $\tS$ by conjugation.
We define the $\sigma$-support $\supp_\sigma(w\tau)$ of $w\tau$ as the smallest $\tau$-stable subset of $\tS$ containing $\supp(w)$ (we should consider the $\tau\sigma$-action in general, but the action of $\sigma$ is trivial in our case because $G$ is split over $F$).
We call an element $w\tau\in W_a\tau$ a $\sigma$-Coxeter element if exactly one simple reflection from each $\tau$-orbit on $\supp_\sigma(w\tau)$ occurs in every (equivalently, any) reduced expression of $w$.

Set $K=G(\cO)$.
For $\alpha\in \Phi$, let $U_\alpha\subseteq G$ denote the corresponding root subgroup. We also set $$I=T(\cO)\prod_{\alpha\in \Phi_+}U_{\alpha}(\vp\cO)\prod_{\beta\in \Phi_-}U_{\beta}(\cO)\subseteq K,$$
which is called the standard Iwahori subgroup associated with the triple $T\subset B\subset G$.
For $J\subsetneq\tS$, let $P_J\supseteq I$ be the standard parahoric subgroup attached to $J$.
By abuse of notation, we also use $P_J$ for the corresponding parahoric group scheme over $\cO$.

In the case $G=\GL_n$, we will use the following description.
Let $T$ be the torus of diagonal matrices and let $B$ be the Borel subgroup of upper triangular matrices.
Let $\chi_{ij}\in X^*(T)$ be given by $\chi_{ij}(\diag(t_1,\ldots,t_n))=t_it_j^{-1}$. Then $\Phi=\{\chi_{ij}\mid i\neq j\}$, $\Phi_+=\{\chi_{ij}\mid i<j\}$, $\Phi_-=\{\chi_{ij}\mid i>j\}$, and $\Pi=\{\chi_{i+1,i}\mid1\leq i<n\}\sqcup\{\chi_{1,n}\}$.
Set $s_i=(i\ i+1)$ for $1\leq i<n$ and $s_0=\vp^{\chi_{1,n}^\vee}(1\ n)$. Then $\bS=\{s_1,\ldots,s_{n-1}\}$ and $\tS=\bS\sqcup\{s_0\}$.
Through the natural isomorphism $X_*(T)\cong\Z^n$, ${X_*(T)}_+$ is identified with the set $\{(m_1,\ldots,m_n)\in\Z^n\mid m_1\geq\cdots\geq m_n\}$.
The finite Weyl group is the symmetric group of degree $n$.
The Iwahori subgroup $I\subset K$ is the inverse image of the lower triangular matrices under the projection $K\rightarrow G(\aFq)$ induced by $\vp\mapsto0$.

Let $\Psi=(\Psi_{ij})\in\GL_{2g}(F)$ be the matrix whose nonzero entries are $\Psi_{i,2g+1-i}=1$ and $\Psi_{2g+1-i,i}=-1$ for $1\leq i\leq g$.
For every $F$-algebra $R$, set
\begin{align*}
\GSp_{2g}(R)&=\{h\in\GL_{2g}(R)\mid {}^th\Psi h=c\Psi\text{ for some }c\in R^\times\},\\
\Sp_{2g}(R)&=\{h\in\GL_{2g}(R)\mid {}^th\Psi h=\Psi\}.
\end{align*}
For $G=\GSp_{2g}$ or $\Sp_{2g}$, we take $T$ and $B$ to be the intersections with the corresponding subgroups of $\GL_{2g}$.
For $G=\GSp_{2g}$, through the natural inclusion $X_*(T)\subseteq\Z^{2g}$, the cocharacter group can be identified with the set $\{(m_1,\ldots,m_{2g})\in\Z^{2g}\mid m_1+m_{2g}=m_2+m_{2g-1}=\cdots=m_g+m_{g+1}\}$, and ${X_*(T)}_+$ with its subset defined by $m_1\geq\cdots\geq m_{2g}$.
In either case, the standard Iwahori subgroup is the intersection of the standard Iwahori subgroup of $\GL_{2g}$ as above with $G$.
For $G=\Sp_{2g}$, the cocharacter group is identified with the subset of $\Z^{2g}$ defined by $m_i+m_{2g+1-i}=0$ for $1\leq i\leq g$.

Assume that $\ch(F)\neq2$.
Let $\Theta=(\Theta_{ij})\in\GL_{2g}(F)$ be the matrix whose nonzero entries are $\Theta_{i,2g+1-i}=1$ for $1\leq i\leq2g$.
For every $F$-algebra $R$, set
\begin{align*}
\operatorname{GSO}_{2g}(R)&=\left\{h\in\GL_{2g}(R)\,\middle|\,{}^th\Theta h=c\Theta\text{ for some }c\in R^\times,\ \det(h)=c^g\right\},\\
\operatorname{SO}_{2g}(R)&=\left\{h\in\GL_{2g}(R)\,\middle|\,{}^th\Theta h=\Theta,\ \det(h)=1\right\}.
\end{align*}
For $G=\operatorname{GSO}_{2g}$ or $\operatorname{SO}_{2g}$, we take $T$ and $B$ to be the intersections with the corresponding subgroups of $\GL_{2g}$.
For $G=\operatorname{GSO}_{2g}$, through the natural inclusion $X_*(T)\subseteq\Z^{2g}$, the cocharacter group can be identified with the set $\{(m_1,\ldots,m_{2g})\in\Z^{2g}\mid m_1+m_{2g}=m_2+m_{2g-1}=\cdots=m_g+m_{g+1}\}$, and ${X_*(T)}_+$ with its subset defined by $m_1\geq\cdots\geq m_g$ and $m_{g-1}+m_g\geq m_1+m_{2g}$.
In either case, the standard Iwahori subgroup is the intersection of the standard Iwahori subgroup of $\GL_{2g}$ as above with $G$.
For $G=\operatorname{SO}_{2g}$, the cocharacter group is identified with the subset of $\Z^{2g}$ defined by $m_i+m_{2g+1-i}=0$ for $1\leq i\leq g$.

\subsection{Affine Deligne--Lusztig Varieties}
\label{ADLV}
Since the underlying topological spaces of the affine flag variety and the affine Grassmannian are Jacobson, any locally closed reduced subscheme is uniquely determined by its $\aFq$-valued points.
Hence we usually identify these spaces and their locally closed reduced subschemes with their sets of $\aFq$-valued points.

Let $B(G)$ denote the set of $\sigma$-conjugacy classes of $G(L)$. 
Thanks to Kottwitz \cite{Kottwitz85}, a $\sigma$-conjugacy class $[b]\in B(G)$ is uniquely determined by two invariants: the Kottwitz point $\kappa(b)\in \pi_1(G)$ and the Newton point $\nu_b\in X_*(T)_{\Q,+}$.
Set $B(G,\mu)=\{[b]\in B(G)\mid \kappa(b)=\kappa(\vp^\mu), \nu_b\pc \mu\}$.

For $w\in \tW$ and $b\in G(L)$, the affine Deligne--Lusztig variety $X_w(b)$ in the affine flag variety $G(L)/I$ is defined as
$$X_w(b)=\{xI\in G(L)/I\mid x^{-1}b\sigma(x)\in IwI\}.$$
For $\mu\in \Y_+$ and $b\in G(L)$, the affine Deligne--Lusztig variety $X_{\mu}(b)$ in the affine Grassmannian $G(L)/K$ is defined as
$$X_{\mu}(b)=\{xK\in G(L)/K\mid x^{-1}b\sigma(x)\in K\vp^{\mu}K\}.$$
Then $X_\mu(b)\neq \emptyset$ if and only if $[b]\in B(G,\mu)$ (see \cite{Gashi10}).
If $\ch(F)>0$, both the affine flag variety and the affine Grassmannian are ind-schemes; if $\ch(F)=0$, they are ind-perfect schemes (see \cite{PR08}, \cite{Zhu17} and \cite{BS17}).
Then the affine Deligne--Lusztig varieties are locally closed subvarieties of them equipped with the reduced scheme structure.
In mixed characteristic, the affine Deligne--Lusztig varieties are locally perfectly of finite type \cite[Lemma 1.1]{HV18}.
For an $F$-algebra $R$, define
$$J_b(R)=\{j\in G(R\otimes_F L)\mid j^{-1}b\sigma(j)=b\},$$
where $\sigma$ acts trivially on $R$.
This defines an algebraic group $J_b$ over $F$.
We set $\J_b=J_b(F)$.
The affine Deligne--Lusztig varieties carry a natural action of $\J_b$ by left multiplication.
Since $b$ is usually fixed in the discussion, we sometimes omit it from the notation.

\begin{lemm}
\label{finite quotient}
Let $Z$ be a quasi-compact subscheme of $X_\mu(b)$, and assume that its stabilizer in $\J_b$ is $\J_b\cap P$ for a parahoric subgroup $P\subset G(L)$. Then the induced action of $\J_b\cap P$ on $Z$ factors through a finite quotient of $\J_b\cap P$.
\end{lemm}
\begin{proof}
Since $Z$ is quasi-compact, it is contained in $\bigcup_{i=1}^r\overline{K\vp^{\lambda_i}K/K}$ for some dominant cocharacters $\lambda_1,\ldots,\lambda_r$.
For sufficiently large $h$, the action of $K=G(\cO)$ on this union factors through $G(\cO/\vp^h\cO)$ (cf.\ \cite[\S2.1.1]{Zhu17}).
Hence the kernel of the $(\J_b\cap P)$-action on $Z$ contains an open subgroup of $\J_b\cap P$.
Since $\J_b\cap P$ is compact, this kernel has finite index.
\end{proof}



The admissible subset of $\tW$ associated with $\mu$ is defined as
$$\Adm(\mu)=\{w\in \tW\mid w\le \vp^{w_0\mu}\ \text{for some}\ w_0\in W_0\}.$$
Set $\SAdm(\mu)=\Adm(\mu)\cap \SW$.
Assume that $\mu$ is minuscule.
Let $\pi\colon G(L)/I\rightarrow G(L)/K$ denote the projection.
Then, by \cite[Theorem 3.2.1]{GH15}, we have
$$X_{\mu}(b)=\bigsqcup_{w\in\SAdm(\mu)}\pi(X_w(b)).$$
This is the {\it Ekedahl--Oort stratification} (cf.\ \cite[\S7.1]{GHN19}).

For a standard parahoric subgroup $P\supseteq I$, let $\overline P$ be the reductive quotient of the special fiber of $P$.
Let $P^+\subseteq I$ denote the pro-unipotent radical of $P$, which coincides with the kernel of the natural surjective map $P\rightarrow \overline P$.
Its restriction to $T(\cO)$ is the usual reduction map $T(\cO)\rightarrow T(\aFq)$.
The quotient $\overline B\coloneqq I/P^+$ is a Borel subgroup of $\overline P$. Thus $P/I$ is naturally identified with the flag variety $\overline P/\overline B$.
Assume that $J\subsetneq\tS$ is $\tau$-stable.
Then $\Ad(\tau)\circ \sigma$ stabilizes $P_J$, $P_J^+$, and $I$, and induces a Frobenius endomorphism of $\overline P_J$ preserving $\overline B_J\coloneqq I/P_J^+$. 
The natural surjective map $P_J\rightarrow\overline P_J$ induces an isomorphism $(\J_\tau\cap P_J)/(\J_\tau\cap P_J^+)\xrightarrow{\sim}\overline P_J^{\Ad(\tau)\circ\sigma}$.


Any two lifts of $w\in \tW$ in $N_{G(L)}T(L)$ are $T(\cO)$-$\sigma$-conjugate (cf.\  \cite[Lemma 2.5]{Gortz19}).
Thus, we fix a lift of $\tau\in \Omega$ in $N_{G(L)}T(L)$ and denote it again by $\tau$.
The following proposition is the key to an explicit description of the affine Deligne--Lusztig varieties.
\begin{prop}
\label{spherical}
Let $\tau\in \Omega$.
Let $w\in W_a\tau$ with $J\coloneqq\supp_\sigma(w)\subsetneq\tS$.
Then $$X_w(\tau)=\bigsqcup_{j\in \J_\tau/\J_\tau\cap P_J} jY(w),$$
where $Y(w)=\{xI\in P_J/I\mid x^{-1}\tau \sigma(x)\in IwI\}$ is a {\it classical Deligne--Lusztig variety} in the flag variety $P_J/I$.
In particular, each $jY(w)$ is an irreducible component.
\end{prop}
\begin{proof}
See \cite[Proposition 2.2.1]{GH15}. See also \cite[Theorem 2]{BR06} for irreducibility.
\end{proof}

Let $w\in \SW\cap W_a\tau$ with $J\coloneqq\supp_\sigma(w)\subsetneq\tS$.
As explained in \cite[\S2.4]{GHN24},
$$\pi(Y(w))=\{xP_{J\cap \bS}\in P_J/P_{J\cap \bS}\mid x^{-1}\tau\sigma(x)\in P_{J\cap \bS}\cdot_\sigma IwI\},$$
where $\cdot_\sigma$ denotes the action by $\sigma$-conjugation.
This is called a {\it fine Deligne--Lusztig variety}.
Moreover, we have
$$\pi(X_w(\tau))=\bigsqcup_{j\in \J_\tau/(\J_\tau\cap P_{J\cup\bS_w})}j\pi(Y(w)),$$
where $\bS_w\coloneqq\max\{K\subseteq\bS\mid \Ad(w)(K)=K\}$.
In particular, if $J$ is a maximal proper $\tau$-stable subset of $\tS$, then $J\cup\bS_w=J$, i.e., $\bS_w\subseteq J$.

\begin{prop}
\label{center}
Let $w\in W_a\tau$ with $J\coloneqq\supp_\sigma(w)\subsetneq\tS$, and let $h\in(\overline P_J)^{\Ad(\tau)\circ\sigma}$. If $h$ fixes $Y(w)$ pointwise, then $h$ lies in the center $Z(\overline P_J)$ of $\overline P_J$. If $w\in\SW$ and $h$ fixes $\pi(Y(w))$ pointwise, then $h$ also lies in $Z(\overline P_J)$.
\end{prop}
\begin{proof}
The first assertion is \cite[Proposition 1.11]{GS26}. We prove the second assertion in a similar way.
Set $H=\overline P_J$, $Q=P_{J\cap\bS}/P_J^+\subseteq H$ and $\delta =\Ad(\tau)\circ \sigma$. We may assume that $J\neq\emptyset$. Since $w\in\SW$, we have $J\cap\bS\subsetneq J$.
Hence $Q$ is a proper parabolic subgroup of $H$.
Again by $w\in\SW$ and the connectedness of the Dynkin diagram of $G$, $\delta$ acts transitively on the set of connected components of $J$.

Passing from the reductive group $H$ over $\F_q$ to its adjoint quotient $H/Z(H)$ does not change the flag variety, the Weyl group $W_J$, or the Deligne--Lusztig varieties. Thus, we may assume that $H$ is $\Fq$-simple and adjoint.
By the closure relation of fine Deligne--Lusztig varieties \cite[Theorem 3.1]{He09}, the closure of $\pi(Y(w))$ in $H/Q$ contains $H^\delta/H^\delta\cap Q$.
Thus it suffices to show that if $h\in H^\delta$ fixes $H^\delta/H^\delta\cap Q$ pointwise, then $h=1$.
Set $Q_0\coloneqq\bigcap_{n\in\Z}\delta^n(Q)$.
Then $Q_0$ is a proper $\delta$-stable parabolic subgroup of $H$, and $H^\delta\cap Q=Q_0^\delta$.
Let $H_1$ be the simple factor corresponding to a connected component of the Dynkin diagram of $H$, and let $d$ be the size of its $\delta$-orbit.
Then $H^\delta\cong H_1^{\delta^d}$ and $Q_0^\delta$ corresponds to $Q_1^{\delta^d}$, where $Q_1\coloneqq Q_0\cap H_1$ is a proper $\delta^d$-stable parabolic subgroup of $H_1$. Hence $H^\delta/Q_0^\delta\cong H_1^{\delta^d}/Q_1^{\delta^d}.$

Therefore, it suffices to show that if $H$ is an absolutely simple adjoint group, $\delta$ is a Frobenius endomorphism of $H$, and $Q\subsetneq H$ is a proper $\delta$-stable parabolic subgroup, then $H^\delta$ acts faithfully on $H^\delta/Q^\delta$. Set $N=\bigcap_{h'\in H^\delta}h'Q^\delta h'^{-1}$. Then $N\triangleleft H^\delta$ and $N\subseteq Q^\delta$. By \cite[Lemma 24.12 \& Corollary 24.13]{MT11}, we have $N\subseteq Z(H^\delta)=Z(H)^\delta=1$, and hence $N=1$.
This finishes the proof.
\end{proof}

\subsection{Deligne--Lusztig Reduction Method}
\label{DL method}
For any $h,h'\in G(L)$, let $\inv(h,h')$ denote the relative position, i.e., the unique element in $\tW$ such that $h^{-1}h'\in I\inv(h,h')I$.
The following Deligne--Lusztig reduction method was established in \cite[Corollary 2.5.3]{GH10}.
\begin{prop}
\label{DL method prop}
Let $w\in \tW$ and let $s\in \tS$ be a simple affine reflection.
If $\ch(F)>0$, then the following two statements hold for any $b\in G(L)$.
\begin{enumerate}[(i)]
\item If $\ell(sws)=\ell(w)$, then there exists a $\J_b$-equivariant universal homeomorphism $X_w(b)\rightarrow X_{sws}(b)$.
\item If $\ell(sws)=\ell(w)-2$, then there exists a decomposition $X_w(b)=X_1\sqcup X_2$ such that
\begin{itemize}
\item $X_1$ is open and there exists a $\J_b$-equivariant morphism $X_1\rightarrow X_{sw}(b)$, which is  the composition of a Zariski-locally trivial $\G_m$-bundle and a universal homeomorphism. 
\item $X_2$ is closed and there exists a $\J_b$-equivariant morphism $X_2\rightarrow X_{sws}(b)$, which is the composition of a Zariski-locally trivial $\A^1$-bundle and a universal homeomorphism. 
\end{itemize}
\end{enumerate}
If $\ch(F)=0$, then the above statements still hold by replacing $\A^1$ and $\G_m$ by $\A^{1,\pfn}$ and $\G_m^{\pfn}$ respectively.
\end{prop}

We sketch the construction of maps in the proposition.
Let $xI\in X_w(b)$.
If $\ell(sw)<\ell(w)$ (we can reduce to this case by exchanging $w$ and $sws$), then let $x_1I$ denote the unique element in $G(L)/I$ such that $\inv(x,x_1)=s$ and $\inv(x_1,b\sigma(x))=sw$.
In the case of (ii), the set $X_1$ (resp.\ $X_2$) above consists of the elements $xI\in X_w(b)$ satisfying $\inv(x_1,b\sigma(x_1))=sw$ (resp.\ $sws$).
Each map in the proposition sends $xI$ to $x_1I$.

We denote by $\delta^+$ the indicator function of the set of positive roots, i.e.,
$$\delta^+\colon \Phi\rightarrow \{0,1\},\quad \alpha \mapsto
\begin{cases}
1 & (\alpha\in \Phi_+) \\
0 & (\alpha\in \Phi_-).
\end{cases}$$
Note that any element $w\in \tW$ can be written in a unique way as $w=x\vp^\mu y$ with $\mu$ dominant, $x,y\in W_0$ such that $\vp^\mu y\in \SW$.
We have $p(w)=xy$ and $\ell(w)=\ell(x)+\la\mu, 2\rho\ra-\ell(y)$.
We define the set of {\it length positive} elements by $$\LP(w)=\{v\in W_0\mid \la v\alpha,y^{-1}\mu\ra+\delta^+(v\alpha)-\delta^+(xyv\alpha)\geq 0\  \text{for all $\alpha\in \Phi_+$}\}.$$
It is easy to check that $y^{-1}\in \LP(w)$ and hence $\LP(w)\neq \emptyset$.

For $w\in \tW$, we say that $w$ has {\it positive Coxeter part} if there exists $v\in \LP(w)$ such that $v^{-1}p(w)v$ is a (partial) Coxeter element.
By \cite[Theorem A]{SSY23} (see also \cite[Theorem 1.1]{HNY22}), this condition induces a simple geometric structure.
Although this is a theorem for general $b$, we specialize it to basic $b$ as follows.
\begin{theo}
\label{simple}
Let $w\in \tW$.
Assume that $\supp_\sigma(w)=\tS$ and there exists $v\in \LP(w)$ such that $v^{-1}p(w)v$ is a Coxeter element with $\supp(v^{-1}p(w)v)=\bS$.
Let $\tau\in\Omega$ be the unique element such that $w\in W_a\tau$.
Then $X_w(\tau)\neq \emptyset$.
Moreover, there exist elements $w=w_0,w_1,\ldots,w_r\in \tW$ satisfying the following conditions:
\begin{enumerate}[(i)]
\item For each $0\leq k<r$, there exists $s\in \tS$ such that $w_{k+1}=sw_ks$ and $\ell(w_{k+1})\leq \ell(w_k)$.
Moreover, $\ell(w_r)<\ell(w_{r-1})$.
\item For each $0\leq k<r$ with $\ell(w_{k+1})=\ell(w_k)-2$, the element $s$ in (i) satisfies $X_{sw_k}(\tau)=\emptyset$ and $X_{sw_ks}(\tau)\neq \emptyset$. In particular, if $k=r-1$, then $s\notin\supp_\sigma(w_r)$.
\item The element $w_r$ is a $\sigma$-Coxeter element and $\J_\tau\cap P_{\supp_\sigma(w_r)}$ is a very special parahoric subgroup of $\J_\tau$.
In particular, $\supp_\sigma(w_r)\subsetneq\tS$ and $w_r$ is of minimal length in its conjugacy class.
\end{enumerate}
Consequently, there are $\J_\tau$-equivariant morphisms
$$X_w(\tau)=X_{w_0}(\tau)\longrightarrow X_{w_1}(\tau)\longrightarrow\cdots\longrightarrow X_{w_r}(\tau),$$
each of which is a universal homeomorphism or a Zariski-locally trivial $\A^1$-fibration up to a universal homeomorphism.
\end{theo}

\begin{proof}
The non-emptiness of $X_w(\tau)$ follows from \cite[Proposition 4.1]{SSY23}.
By \cite[Corollary 2.10]{HN14} and \cite[Theorem 5.7(c)]{SSY23}, there exists a sequence satisfying (i) and (ii) such that $w_r$ is of minimal length in its conjugacy class (cf.\ Proposition \ref{spherical} for the last assertion in (ii)).
By \cite[Theorem 5.7(d) \& Theorem 5.8]{SSY23}, the subset $\supp_\sigma(w_r)$ is very special with respect to $\tau$ and $w_r$ is a $\sigma$-Coxeter element.
The final assertion follows from Proposition \ref{DL method prop}.
\end{proof}

\begin{rema}
\label{positive Coxeter full support}
Assume that $X_w(\tau)\neq \emptyset$ and $\supp_\sigma(w)=\tS$.
If $w$ has positive Coxeter part, then any partial Coxeter element $v^{-1}p(w)v$ with $v\in \LP(w)$ satisfies $\supp(v^{-1}p(w)v)=\bS$ (cf.\  \cite[Proposition 5]{Schremmer23}).
\end{rema}

Combining Theorem \ref{simple} with Proposition \ref{spherical}, we see that, when $w\in W_a\tau$ is of positive Coxeter type, $\J_\tau$ acts transitively on the irreducible components of $X_w(\tau)$, and each irreducible component is an iterated $\A^1$-fibration over a classical Deligne--Lusztig variety of Coxeter type, up to universal homeomorphisms.
The stabilizer of each irreducible component is conjugate in $\J_\tau$ to $\J_\tau\cap P_{\supp_\sigma(w_r)}$.
In fact, by \cite[Theorem 5.20]{SSY23}, each irreducible component is universally homeomorphic to a product $Y(w_r)\times\A^d$, where $Y(w_r)$ is the classical Deligne--Lusztig variety of Coxeter type appearing in Proposition \ref{spherical}.
This product description has been generalized by Nie, Schremmer and Yu \cite{NSY25} to arbitrary $b$ with $X_w(b)\neq\emptyset$, in which case $\Gm$-factors may also occur.
We do not use this product decomposition in this paper.

\begin{lemm}
\label{injective}
Let $w\in\SW$, and assume that there exists $v\in\LP(w)$ such that $v^{-1}p(w)v$ is a Coxeter element with $\supp(v^{-1}p(w)v)=\bS$. Then, for any $b\in G(L)$, the natural map
$X_w(b)\longrightarrow\pi(X_w(b))$
is bijective.
\end{lemm}
\begin{proof}
Let $x\in W_0$ be such that $xw=wx$. By \cite[Lemma 2.1]{Shimada4}, it suffices to show that $x=1$. By \cite[Lemma 2.1]{He23}, we have $x\in W_{\bS_w}$.
If $\bS_w\neq\emptyset$, then $\Ad(w)$ permutes $\bS_w$, and the assumption $w\in\SW$ implies that $p(w)$ permutes the corresponding simple roots. Their sum is then a nonzero $p(w)$-fixed vector, which is impossible since $p(w)$ is conjugate to a Coxeter element. Hence $\bS_w=\emptyset$ and $x=1$.
\end{proof}
Let $\mu\in\Y_+$ be non-central and minuscule, and let
$\tau_\mu\in\Omega$ be the image of $\vp^\mu$ under the natural
projection $\tW\rightarrow\Omega$.
As before, we fix a lift of $\tau_\mu$, which we denote by the same symbol.
Setting $b=\tau_\mu$, we obtain the following stronger statement.

\begin{coro}
\label{projection universal homeomorphism}
If $w\in\SAdm(\mu)$ satisfies the assumptions of Lemma \ref{injective}, then the natural map $X_w(\tau_\mu)\rightarrow\pi(X_w(\tau_\mu))$ is a universal homeomorphism.
\end{coro}
\begin{proof}
This follows from Lemma \ref{injective}, the properness of $\pi$, and \cite[Lemma 2.4]{ST24}.
\end{proof}
\begin{rema}
\label{dense Coxeter stratum}
It is easy to see that there exists a Coxeter element $v\in W_0$ satisfying $\supp(v)=\bS$ and $\vp^\mu v\in\SAdm(\mu)$. By \cite[Theorem 5.7(b)]{SSY23}, we have $\dim X_{\vp^\mu v}(\tau_\mu)=\dim X_\mu(\tau_\mu)$. The variety $X_\mu(\tau_\mu)$ is equidimensional (cf.\ \cite{HV12}, \cite{Takaya25}), and since $G$ is split, $\J_{\tau_\mu}$ acts transitively on its irreducible components (cf.\ \cite{HV18}). Thus, every irreducible component of $X_\mu(\tau_\mu)$ contains an irreducible component of $\pi(X_{\vp^\mu v}(\tau_\mu))$ as an open dense subscheme. 
The irreducible components of $\pi(X_{\vp^\mu v}(\tau_\mu))$ are pairwise disjoint, and they are affine because they are Deligne--Lusztig varieties of Coxeter type (cf.\ \cite[Corollary 2.8]{Lusztig76}). Hence, if $\dim X_{\mu} (\tau_{\mu}) >0$, then $X_\mu(\tau_\mu)\setminus\pi(X_{\vp^\mu v}(\tau_\mu))(\neq \emptyset)$ is equidimensional of codimension $1$ in $X_\mu(\tau_\mu)$. In particular, such $v$ is unique.

If $g\geq3$, $G$ is either $\GL_{2g}$ or $\GSp_{2g}$, and $\mu=(1^{(g)},0^{(g)})$, then $\supp_\sigma(\vp^\mu v)=\tS$, since $(G,\mu)$ is not fully Hodge--Newton decomposable (cf.\ \cite[Theorem D]{GHN19}).
\end{rema}

\subsection{The $\J$-stratification}
\label{J-stratification}
We recall the $\J$-stratification introduced by Chen and Viehmann in \cite{CV18}.
More precisely, following \cite{Shimada6}, we use a slightly finer version of this stratification.
We retain the name $\J$-stratification, as the two stratifications coincide in important cases.
\label{J-str}

Let $\inv_K(h,h')$ denote the image of $\inv(h,h')$ under the natural projection $\tW\rightarrow \tW/W_0$.
By sending $\vp^\ld u\in \tW$ to $\ld\in \Y$, we identify $\tW/W_0$ with $\Y$.
We assign to an element $xK\in G(L)/K$ the function
$$f=f_x\colon \J_b\rightarrow \Y,\quad j\mapsto \inv_K(j,x).$$
Note that $f$ is constant on cosets $j(I\cap \J_b)$.
By \cite[Theorem 2.10]{Gortz19}, the set $$S_f\coloneqq \{xK\in G(L)/K\mid f_x=f\}=\bigcap_{j\in \J_b}jI\vp^{f(j)}K/K$$ defines a locally closed reduced $\aFq$-subscheme of $\cG r := G(L)/K$.
Hence this invariant induces a stratification of $\cG r$.
By definition, each $I\vp^\ld K/K$ is a union of strata $S_f$.
By intersecting $S_f$ with $X_\mu(b)$, we obtain the $\J_b$-stratification of $X_\mu(b)$.
For $w\in \tW$, the $\J_{\dot w}$-stratification is independent of the choice of lift $\dot w$ in $N_{G(L)}T(L)$.
We usually fix $b\in \Omega$ and omit it from the notation.

Assume that $\mu$ is minuscule, and let $\tau\in\Omega$ satisfy $[\tau]\in B(G,\mu)$.
For $\lambda\in\Y$, set $\Xl=X_\mu(\tau)\cap I\vp^\lambda K/K$ and $\cAm=\{\lambda\in\Y\mid\Xl\neq\emptyset\}$. By \cite[Proposition 2.9]{Nie22}, we have $\cAm=\{\lambda\in\Y\mid\tau(\lambda)-\lambda\in W_0\mu\}$. Let $\cSm$ be the set of non-empty $\J$-strata of $X_\mu(\tau)$. We call a $\J$-stratum {\it parahoric} (resp.\ {\it standard parahoric}) if its stabilizer in $\J_\tau$ is parahoric (resp.\ standard parahoric), and denote by $\cSm^{\mathrm{par}}$ the set of parahoric $\J$-strata. For $\alpha\in\Phi$, set $\lambda_\alpha=\langle\alpha,\lambda\rangle-1$ if $\alpha\in\Phi_+$ and $\lambda_\alpha=\langle\alpha,\lambda\rangle$ if $\alpha\in\Phi_-$. Then $\lambda\in\cAm$ is called {\it small} if, for every $\alpha\in\Pi$, there exists an element $\beta$ in the $p(\tau)$-orbit of $\alpha$ such that $\lambda_\beta\leq0$. Let $\cAm^{\mathrm{sm}}\subseteq\cAm$ be the set of small cocharacters. 
Set $\Pi(\lambda)=\{\alpha\in\Pi\mid\lambda_\beta\geq0\text{ for every $\beta$ in the $p(\tau)$-orbit of $\alpha$}\}$.
The following is \cite[Theorem 4.5 \& Corollary 4.6]{Shimada6}.
\begin{theo}
\label{parahoric J-strata}
For $\lambda\in\cAm$, the following are equivalent:
\begin{enumerate}[(i)]
\item The cocharacter $\lambda$ is small;
\item The variety $\Xl$ is irreducible;
\item The variety $\Xl$ contains a unique standard parahoric $\J$-stratum $S$.
\end{enumerate}
If these conditions hold, then $S$ is open in $\Xl$, and the stabilizers in $\J_\tau$ of $S$ and $\overline{\Xl}$ are both $\J_\tau\cap P_{\Pi(\lambda)}$.
Moreover, there is a bijection
$$\J_\tau\backslash\cSm^{\mathrm{par}}\xrightarrow{\sim}\Omega\backslash\cAm^{\mathrm{sm}},\qquad \J_\tau S\longmapsto\Omega\lambda,$$
where $X_\mu^\lambda(\tau)$ contains a standard parahoric representative of $\J_\tau S$.
\end{theo}

For $\lambda\in\cAm$, let $\epsilon_\lambda\in W_0$ be the unique element such that $\{\alpha\in\Phi\mid\lambda_\alpha\geq0\}=\epsilon_\lambda\Phi_+$, and set $\lambda^\natural=\tau(\lambda)-\lambda$ and $\lambda^\flat=\epsilon_\lambda^{-1}\lambda^\natural$. By \cite[Proposition 3.9]{Shimada6}, there is a map
$$\flat\colon\Omega\backslash\cAm^{\mathrm{sm}}\longrightarrow\{\nu\in W_0\mu\mid\nu\leq\nu_\tau\},\qquad\Omega\lambda\longmapsto\lambda^\flat.$$
Here $\nu\leq\nu_\tau$ means that $\nu_\tau-\nu$ is a non-negative rational linear combination of positive coroots.
By \cite[Proposition 3.12]{Shimada6}, we have 
\[
\label{eqn:dimension}
\dim X_\mu^\lambda(\tau)=\langle\rho,\mu+\lambda^\flat\rangle.
\]
In general, only the equality $\#\bigl(\Omega\backslash\cAm^{\mathrm{sm}}\bigr)=\#\{\nu\in W_0\mu\mid\nu\leq\nu_\tau\}$ is known. Fortunately, for general linear and symplectic similitude groups, we have the following result \cite[Theorem 5.18 \& \S 6.2]{Shimada6}. 
\begin{theo}
\label{flat bijection}
If $G=\GL_{n}$ or $\GSp_{2g}$, then the map $\flat$ is bijective.
\end{theo}

\begin{rema}
Karemaker and Yu introduced a stratification of the union of all supersingular Ekedahl--Oort strata. It has a unique maximal stratum, which is open and dense in this union. If $g$ is even and $p\geq5$, every point of this stratum has automorphism group $\{\pm1\}$. They conjectured that this stratification refines the $\J$-stratification. See \cite[Theorem 6.17 \& Remark 6.21]{KY24}.
\end{rema}

\subsection{Reformulation}
Let $\bQp$ be the completion of the maximal unramified extension of $\Q_p$, and let $\breve{\Z}_p$ be its ring of integers. From this subsection until the end of this paper, we assume that $F=\Q_p$ and $L=\bQp$.
The following reformulation is due to Viehmann (see \cite[\S1]{Viehmann08b} and \cite[\S2]{Viehmann26}).

Let $(\mathbb X,\vph_0)$ be a principally polarized superspecial $p$-divisible group of dimension $g$ over $\overline{\F}_p$. For any $\breve{\Z}_p$-scheme $S$ on which $p$ is locally nilpotent, let
$$\mathcal M_g(S)=\{(X,\vph,\rho)\}/\cong,$$
where $X$ is a $p$-divisible group over $S$, $\vph$ is a $p$-power multiple of a principal polarization of $X$, and $\rho\colon\mathbb X_{\overline S}\rightarrow X_{\overline S}$ is a quasi-isogeny over the reduction $\overline S$ of $S$ modulo $p$ compatible with the polarizations. This functor is represented by a formal scheme $\mathcal M_g$ locally formally of finite type over $\operatorname{Spf}\breve{\Z}_p$. By abuse of notation, we identify $\mathcal M_g$ with its underlying reduced scheme over $\overline{\F}_p$.

A point $x\in\mathcal M_g(\overline{\F}_p)$ is represented by a triple $(X,\vph,\rho)$. The Dieudonn\'e module corresponding to $(X,\vph,\rho)$ is a triple $(M,F,\langle\cdot,\cdot\rangle)$, where $M$ is a free $\breve{\Z}_p$-module of rank $2g$, $F\colon M\rightarrow M$ is $\sigma$-linear with $M\supseteq F(M)\supseteq pM$, and $\langle\cdot,\cdot\rangle$ is a symplectic pairing on $M\otimes_{\breve{\Z}_p}\bQp$ such that $M^\vee=cM$ for some $c\in\bQp$, where
$$M^\vee=\{v\in M\otimes_{\breve{\Z}_p}\bQp\mid\langle v,M\rangle\subseteq\breve{\Z}_p\}.$$
Moreover, $F$ preserves the pairing up to a scalar, and
\begin{align*}
\operatorname{Aut}(X,\vph)&=\operatorname{Aut}(M,F,\langle\cdot,\cdot\rangle)\\
&=\{j\in\operatorname{Aut}_{\breve{\Z}_p}(M)\mid jF=Fj,\ \langle jv,jw\rangle=\langle v,w\rangle\text{ for all }v,w\in M\}.
\end{align*}

Let $(N,F,\langle\cdot,\cdot\rangle)$ be the rational Dieudonn\'e module of $(\mathbb X,\vph_0)$ and set $\breve N=N\otimes_{\Q_p}\bQp$. There is a basis $e_1,\ldots,e_g,f_1,\ldots,f_g$ of $N$ such that $F(e_i)=f_i$, $F(f_i)=pe_i$, and
$$\langle e_i,f_{g+1-i}\rangle=-\langle f_i,e_{g+1-i}\rangle=1$$
for all $i$, with all other pairings between the basis vectors equal to zero. Via $\rho$, the module $M$ is identified with a lattice in $\breve N$, and
\begin{align*}
\operatorname{Aut}(M,F,\langle\cdot,\cdot\rangle)&\subseteq\operatorname{Aut}(\breve N,F,\langle\cdot,\cdot\rangle)\\
&=\{j\in\operatorname{Aut}_{\bQp}(\breve N)\mid jF=Fj,\ \langle jv,jw\rangle=\langle v,w\rangle\text{ for all }v,w\in\breve N\}.
\end{align*}
The group on the right acts on $\mathcal M_g$ by precomposing $\rho$.
The automorphism groups of points in the same orbit are conjugate subgroups of this group.

Let $G=\GSp_{2g}$, $\mu=(1^{(g)},0^{(g)})\in X_*(T)$ and $\tau=\tau_\mu=\begin{pmatrix}0& p\cdot 1_g\\1_g&0\end{pmatrix}\in\Omega$.
We have the bijection
$$X_\mu(\tau)\xrightarrow{\sim}\mathcal M_g(\overline{\F}_p),\qquad xK\longmapsto x\breve{\Z}_p^{2g}.$$
More precisely, the perfection of $\mathcal M_g$ is isomorphic to $X_\mu(\tau)$ as a perfect scheme (cf.\ \cite[Proposition 5.1]{GHN24}).
The above identification of $\breve N$ with $\bQp^{2g}$ induces an isomorphism between $\operatorname{Aut}(\breve N,F,\langle\cdot,\cdot\rangle)$ and $\J^1\coloneqq\J \cap\Sp_{2g}(\bQp)$, where $\J=\J_\tau$.

Recall that $\mathscr S_g$ is the supersingular locus in the moduli space over $\overline{\F}_p$ of $g$-dimensional principally polarized abelian varieties with a level $N$ structure for some integer $N$ coprime to $p$, and write $\mathscr S_g^{\pfn}$ for the perfection of $\mathscr S_g$. Let $K^p\subseteq G(\A_f^p)$ be the principal congruence subgroup of level $N$. Fix a principally polarized superspecial abelian variety whose polarized $p$-divisible group is isomorphic to $(\mathbb X,\vph_0)$, and let $I$ be the algebraic group over $\Q$ of its self-quasi-isogenies that preserve the polarization up to a scalar, so that $I_{\Q_p}\cong J_\tau$ and $I_{\A_f^p}\cong G_{\A_f^p}$. If $N\geq3$, the Rapoport--Zink uniformization gives an isomorphism of perfect schemes
$$I(\Q)\backslash\bigl(X_\mu(\tau)\times G(\A_f^p)/K^p\bigr)\xrightarrow{\sim}\mathscr S_g^{\pfn}$$
(cf.\ \cite[Theorem 6.30]{RZ96} and \cite[\S5.2]{HZZ21}).
The morphism forgetting the level structure is finite and surjective, hence a quotient map. 
Thus the presence of a level structure does not affect the openness, density or dimension of the subvarieties considered in this paper.
The quotient morphism from $X_\mu(\tau)\times G(\A_f^p)/K^p$ to its $I(\Q)$-quotient is \'etale and surjective. 
Consequently, the openness, density and dimension of $I(\Q)$-stable locally closed subvarieties may be checked before taking the quotient.

Let $(A,\vph)$ be a principally polarized supersingular abelian variety over $\overline{\F}_p$ and let $(X,\vph)$ be its $p$-divisible group. Then $\operatorname{Aut}(A,\vph)$ is a finite subgroup of $\operatorname{Aut}(X,\vph)$. 
Under the uniformization theorem, the pair $(A,\vph)$ corresponds to the class of a pair $((X,\vph,\rho),g^pK^p)$ for some $\rho$ and $g^pK^p$.
It is therefore enough to study finite-order automorphisms of $(X,\vph)$.
By the above identification of $\mathcal M_g$ with $X_\mu(\tau)$, Oort's conjecture follows from the following:
\begin{theo}
Let $g\ge 2$ and $p\ge 5$.
There exists an open dense subscheme $U\subseteq X_\mu(\tau)$ such that, for every closed point $xK\in U$, if $j\in\J^1$ has finite order and $jxK=xK$, then $j=\pm1$.
\end{theo}
In this paper, we prove this theorem in various ways. See Remark \ref{Oort conjecture from a-number loci}.

We have $\J\cong\GSp_g(D)$, where $D$ is the division algebra over $\Q_p$ of invariant $\frac{1}{2}$.
Let $\cO_D$ be the unique maximal order of $D$ and let $\varpi_D$ be a uniformizer of $\cO_D$.
A key ingredient in the proof is the following lemma, proved in \cite[Lemma 6.16]{KY24}.
\begin{lemm}
\label{lemm:torsion-free}
Let $g\geq 1$ and $s\geq 1$ be integers, and set $V_{p,s}\coloneqq 1+\varpi_D^s\operatorname{Mat}_g(\cO_D)\subseteq\GL_g(\cO_D)$. Then $V_{p,s}$ is torsion-free if and only if either (1) $s\geq 3$, (2) $p\geq 3$ and $s=2$, or (3) $p\geq 5$ and $s=1$.
\end{lemm}

\section{Ekedahl--Oort strata of positive Coxeter type}

For each subgroup $G\subseteq\GL_{2g}$ considered in this section, we take $T$ and $B$ to be its intersections with the diagonal torus and the Borel subgroup of upper triangular matrices in $\GL_{2g}$, respectively.
Let $\mu=(1^{(g)},0^{(g)})\in X_*(T)$, $\tau=\tau_\mu=\begin{pmatrix}0& p\cdot 1_g\\1_g&0\end{pmatrix}\in\Omega$ and $\J=\J_\tau$.
For the computation of very special parahoric subgroups, see \cite[Proposition 2.2.5]{HZZ21}. See also Remark \ref{positive Coxeter full support}. For $x\in\overline{\mathbb F}_p$, let $[x]\in\breve{\Z}_p$ denote its Teichm\"uller lift.
\subsection{The case of $\GL_{2g}$}
Let $G=\GL_{2g}$.
The goal of this subsection is to prove the following theorem.
\begin{theo}
\label{main theo GL}
Assume that $w\in\SAdm(\mu)$ has positive Coxeter part, $\supp_\sigma(w)=\tS$, $g\geq 3$ and $p\geq 5$.
There exists an open dense subscheme $U\subseteq\pi(X_w(\tau))$ such that, for every closed point $xK\in U$, if $j\in\J$ has finite order and $jxK=xK$, then $j$ is a unit-root scalar in $\Z_p^\times$.
\end{theo}
By Remark \ref{dense Coxeter stratum}, this is a generalization of \cite[Theorem 5.1]{Viehmann26}. Note that the unit-root scalars in $\Z_p^\times$ are the Teichm\"uller lifts of the elements of $\mathbb F_p^\times$.

We have $\J\cong\GL_g(D)$.
For $1\leq i<2g$, let $s_i$ be the permutation matrix associated with $(i\ i+1)$, and set
$$s_0=\begin{pmatrix}0&0&p\\0&1_{2g-2}&0\\p^{-1}&0&0\end{pmatrix}.$$
We denote their images in $\tW$ by the same symbols.
Then $\bS=\{s_1,s_2,\ldots,s_{2g-1}\}$ and $\tS=\bS\cup\{s_0\}$.
Moreover, $\tau s_i\tau^{-1}=s_{i+g}$, where the subscript is taken modulo $2g$.

\begin{lemm}
\label{lemm:GLveryspecial}
Assume that $J$ is a very special subset of $\tS$ with respect to $\tau$; equivalently, for some integer $0\leq i<g$, we have $J=\tS\setminus\{s_i,s_{i+g}\}$. As a connected reductive group over $\mathbb F_p$, $\overline P_J$ is isomorphic to
$\operatorname{Res}_{\mathbb F_{p^2}/\mathbb F_p}\GL_g.$
\end{lemm}
\begin{proof}
The absolute root datum obtained by deleting $\{s_i,s_{i+g}\}$ is that of $\GL_g\times\GL_g$. Since $\Ad(\tau)\circ\sigma$ exchanges the two factors, the assertion follows immediately.
\end{proof}

Writing $D=\mathbb Q_{p^2}\oplus\varpi_D\mathbb Q_{p^2}$ and $\cO_D=\mathbb Z_{p^2}\oplus\varpi_D\mathbb Z_{p^2}$, with $\varpi_D^2=p$ and $\varpi_Da=\sigma(a)\varpi_D$ for $a\in\mathbb Q_{p^2}$, the isomorphism $\GL_g(D)\xrightarrow{\sim}\J$ is given by
$$X+\varpi_DY\longmapsto\begin{pmatrix}X&p\sigma(Y)\\Y&\sigma(X)\end{pmatrix}$$
where $X,Y\in\operatorname{Mat}_g(\mathbb Q_{p^2})$.
For $J=\tS\setminus\{s_0,s_g\}$, this isomorphism identifies $\J\cap P_J$ with $\GL_g(\cO_D)$ and $\J\cap P_J^+$ with $V_{p,1}$. In this case, we have
$$P_J^+=\left\{\begin{pmatrix}1_g+pA&pB\\ C&1_g+pD\end{pmatrix}\mathrel{\Big|}A,B,C,D\in\operatorname{Mat}_g(\breve{\Z}_p)\right\}.$$
For $J=\tS\setminus\{s_i,s_{i+g}\}$, the groups $\J\cap P_J$ and $\J\cap P_J^+$ are conjugate in $\J$ to $\GL_g(\cO_D)$ and $V_{p,1}$, respectively.
For $s\in\{s_0,s_g\}$ and $x\in\overline{\mathbb F}_p$, set $u(x)=1_{2g}+p[x]E_{1,2g}$ if $s=s_0$ and $u(x)=1_{2g}+[x]E_{g+1,g}$ if $s=s_g$, where $E_{k,l}$ is the matrix whose $(k,l)$-entry is $1$ and whose other entries are $0$.

\begin{proof}[Proof of Theorem \ref{main theo GL}]
By Corollary \ref{projection universal homeomorphism} and the $\J$-equivariance of $\pi$, it suffices to prove the same assertion for $X_w(\tau)$.
Let $w_1,\ldots,w_r$ be as in Theorem \ref{simple}.
By Proposition \ref{spherical} and Theorem \ref{simple}, there is an irreducible component $C$ of $X_w(\tau)$ whose stabilizer in $\J$ is $\J\cap P_J$, where $J=\supp_\sigma(w_r)$. Since $J$ is very special with respect to $\tau$, we have $J=\tS\setminus\{s_i,s_{i+g}\}$ for some $0\leq i<g$.
Set $\omega=\begin{pmatrix}0&p\\1_{2g-1}&0\end{pmatrix}$, so that $\omega^g=\tau$, $\omega^{-i}s_i\omega^i=s_0$, and $\omega^{-i}s_{i+g}\omega^i=s_g$. For each $0\leq k\leq r$, conjugation by $\omega^{-i}$ induces an isomorphism
$$X_{w_k}(\tau)\xrightarrow{\sim}X_{\omega^{-i}w_k\omega^i}(\tau),\qquad xI\longmapsto\omega^{-i}x\omega^iI.$$
This isomorphism is not $\J$-equivariant, but it identifies the action of $j\in\J$ with that of $\omega^{-i}j\omega^i\in\J$.
Replacing $X_{w_k}(\tau)$ by $X_{\omega^{-i}w_k\omega^i}(\tau)$ for $0\leq k\leq r$, we may therefore assume that $i=0$ and hence $J=\tS\setminus\{s_0,s_g\}$.

Since $\J$ acts transitively on the irreducible components of $X_w(\tau)$, it suffices to find a non-empty open subscheme $U\subseteq C$ such that, for every closed point $xI\in U$, if $j\in\J$ has finite order and $jxI=xI$, then $j$ is a unit-root scalar in $\Z_p^\times$.
Note that every scalar in $\Z_p^\times$ fixes $C$ pointwise.
For each $j\in\J\cap P_J$, the locus $\{xI\in C\mid jxI=xI\}$ is closed in $C$. Indeed, locally closed subvarieties of the affine flag variety are separated.
By Lemma \ref{finite quotient}, the action of $\J\cap P_J$ on $C$ factors through a finite quotient. Thus, it remains to show that if $j\in\J\cap P_J$ has finite order and fixes $C$ pointwise, then $j$ is a unit-root scalar in $\Z_p^\times$.

Let $j$ be such an element. By Proposition \ref{spherical} and Theorem \ref{simple}, the element $j$ fixes $Y(w_r)$ pointwise, where $Y(w_r)=\{xI\in P_J/I\mid x^{-1}\tau\sigma(x)\in Iw_rI\}$. Let $\overline j\in \GL_g(\mathbb F_{p^2})$ be the image of $j$ under the isomorphisms
$$(\J\cap P_J)/(\J\cap P_J^+)\cong(\overline P_J)^{\Ad(\tau)\circ\sigma}\cong\GL_g(\mathbb F_{p^2}).$$
By Proposition \ref{center}, we have $\overline j=c$ for some $c\in\mathbb F_{p^2}^\times$.
Set $z=\diag([c],\ldots,[c],[c^p],\ldots,[c^p])$, where each of $[c]$ and $[c^p]$ occurs $g$ times.
Then $z\in\J\cap T(\breve{\Z}_p)\subseteq\J\cap P_J$, and its image under the above isomorphism is $c$. Hence $j=zu^+$ for some $u^+\in\J\cap P_J^+$.

Let $s\in\tS$ be the element corresponding to $k=r-1$ in Theorem \ref{simple} (i), so that $w_r=sw_{r-1}s$ and $\ell(w_r)=\ell(w_{r-1})-2$. By Theorem \ref{simple} (ii), we have $s=s_0$ or $s=s_g$. The map $x\mapsto u(x)sI$ defines an isomorphism $\mathbb A^1\xrightarrow{\sim}IsI/I$. Let $f\colon X_{w_{r-1}}(\tau)\rightarrow X_{w_r}(\tau)$ be the morphism induced by Proposition \ref{DL method prop}. For $yI\in Y(w_r)$, we have
$$f^{-1}(yI)=\{yu(x)sI\mid x\in\overline{\mathbb F}_p\}.$$
Note that the right-hand side is independent of the choice of $y$, and $yu(x)sI=yu(x')sI$ if and only if $x=x'$.
By assumption and Theorem \ref{simple}, $j$ fixes $f^{-1}(yI)$ pointwise.
Since $P_J^+$ is normal in $P_J$ and the image of $z$ in $P_J/P_J^+$ is central, both $y^{-1}u^+y$ and $z^{-1}y^{-1}zy$ lie in $P_J^+$. Hence
$$y^{-1}u^+y=\begin{pmatrix}1_g+pA&pB\\ C&1_g+pD\end{pmatrix},\qquad z^{-1}y^{-1}zy=\begin{pmatrix}1_g+pA'&pB'\\ C'&1_g+pD'\end{pmatrix}$$
for some $A,B,C,D,A',B',C',D'\in\operatorname{Mat}_g(\breve{\Z}_p)$.
Thus
$$u^+yu(x)sI=
\begin{cases}
yu(x+\overline{b_{1,g}})sI & \text{if }s=s_0,\\
yu(x+\overline{c_{1,g}})sI & \text{if }s=s_g,
\end{cases}$$
where $\overline{b_{1,g}}$ and $\overline{c_{1,g}}$ are the reductions modulo $p$ of the $(1,g)$-entries of $B$ and $C$, respectively.
Moreover, $zu(x)z^{-1}=u(c^{1-p}x)$ if $s=s_0$ and $zu(x)z^{-1}=u(c^{p-1}x)$ if $s=s_g$. Therefore
$$zyu(x)sI=
\begin{cases}
yu(c^{1-p}(x+\overline{b'_{1,g}}))sI & \text{if }s=s_0,\\
yu(c^{p-1}(x+\overline{c'_{1,g}}))sI & \text{if }s=s_g,
\end{cases}$$
where $\overline{b'_{1,g}}$ and $\overline{c'_{1,g}}$ are the reductions modulo $p$ of the $(1,g)$-entries of $B'$ and $C'$, respectively. It follows that
$$jyu(x)sI=
\begin{cases}
yu(c^{1-p}(x+\overline{b_{1,g}}+\overline{b'_{1,g}}))sI & \text{if }s=s_0,\\
yu(c^{p-1}(x+\overline{c_{1,g}}+\overline{c'_{1,g}}))sI & \text{if }s=s_g.
\end{cases}$$
Note that $\overline{b_{1,g}}$, $\overline{c_{1,g}}$, $\overline{b'_{1,g}}$ and $\overline{c'_{1,g}}$ are independent of $x$.
Since $j$ fixes $f^{-1}(yI)$ pointwise, we have, for every $x\in\overline{\mathbb F}_p$,
$$\begin{aligned}
c^{1-p}(x+\overline{b_{1,g}}+\overline{b'_{1,g}})&=x &&\text{if }s=s_0,\\
c^{p-1}(x+\overline{c_{1,g}}+\overline{c'_{1,g}})&=x &&\text{if }s=s_g.
\end{aligned}$$
The equations for $x=0$ and $x=1$ imply that $c^{p-1}=1$, and hence $c\in\mathbb F_p^\times$.

Thus $z=[c]\in\Z_p^\times$, and $u^+=z^{-1}j$ has finite order. Since $\J\cap P_J^+$ is torsion-free by $p\ge 5$ and Lemma \ref{lemm:torsion-free}, we have $u^+=1$ and $j=z\in\Z_p^\times$ as desired.
\end{proof}

\subsection{The case of $\GSp_{2g}$}
Let $G=\GSp_{2g}$.
The main goal of this subsection is to prove the following theorem.
\begin{theo}
\label{main theo}
Assume that $w\in\SAdm(\mu)$ has positive Coxeter part, $\supp_\sigma(w)=\tS$, $g\ge 3$ and $p\geq 5$.
There exists an open dense subscheme $U\subseteq\pi(X_w(\tau))$ such that, for every closed point $xK\in U$, if $j\in \J^1$ has finite order and $jxK=xK$, then $j=\pm 1$.
\end{theo}
By Remark \ref{dense Coxeter stratum}, this is a generalization of \cite[Theorem 2.2]{Viehmann26}.
The set $\SAdm(\mu)$ parametrizes the global Ekedahl--Oort strata, and the intersection of the stratum attached to $w$ with the supersingular locus is uniformized by $\pi(X_w(\tau))$ (cf.\ \cite[\S3.2 \& \S5.1]{Wang21}).
In particular, this intersection is non-empty if and only if $X_w(\tau)\neq\emptyset$ (cf.\ \cite[Lemma 7.6]{GHN19}).
The Ekedahl--Oort stratum attached to $w$ is contained in the supersingular locus if and only if $\supp_\sigma(w)\subsetneq\tS$ by \cite[Proposition 5.2]{Wang21}.

For a $\Q_p$-algebra $R$, define
$$J_\tau^1(R)=\{j\in\Sp_{2g}(R\otimes_{\Q_p}\bQp)\mid j^{-1}\tau\sigma(j)=\tau\},$$
where $\sigma$ acts trivially on $R$.
This defines a reductive group $J_\tau^1$ over $\Q_p$, which is an inner form of $\Sp_{2g}$, and $\J^1=J_\tau^1(\Q_p)$.

Assume that $J\subsetneq\tS$ is $\tau$-stable, and let $P_J$ be the standard parahoric subgroup of $G=\GSp_{2g}$ attached to $J$.
The groups $P_J\cap\Sp_{2g}(\bQp)$ and $P_J^+\cap\Sp_{2g}(\bQp)$ are respectively the standard parahoric subgroup of $\Sp_{2g}(\bQp)$ attached to $J$ and its pro-unipotent radical.
Let $\overline P_J^1$ be the reductive quotient of the special fiber of $P_J\cap\Sp_{2g}(\bQp)$, so that $(P_J\cap\Sp_{2g}(\bQp))/(P_J^+\cap\Sp_{2g}(\bQp))\xrightarrow{\sim}\overline P_J^1$.
Let $\overline B_J^1\coloneqq(I\cap\Sp_{2g}(\bQp))/(P_J^+\cap\Sp_{2g}(\bQp))$, which is a Borel subgroup of $\overline P_J^1$. Thus $(P_J\cap\Sp_{2g}(\bQp))/(I\cap\Sp_{2g}(\bQp))$ is naturally identified with the flag variety $\overline P_J^1/\overline B_J^1$.
Then $\Ad(\tau)\circ\sigma$ stabilizes $P_J\cap\Sp_{2g}(\bQp)$, $P_J^+\cap\Sp_{2g}(\bQp)$, and $I\cap\Sp_{2g}(\bQp)$, and induces a Frobenius endomorphism of $\overline P_J^1$ preserving $\overline B_J^1$.
The natural surjective map $P_J\cap\Sp_{2g}(\bQp)\rightarrow\overline P_J^1$ induces an isomorphism $(\J^1\cap P_J)/(\J^1\cap P_J^+)\xrightarrow{\sim}(\overline P_J^1)^{\Ad(\tau)\circ\sigma}$.


For $1\leq i<g$, let $s_i$ be the permutation matrix associated with $(i\ i+1)(2g-i\ 2g+1-i)$.
Set
$$s_g=\begin{pmatrix}1_{g-1}&0&0&0\\0&0&1&0\\0&-1&0&0\\0&0&0&1_{g-1}\end{pmatrix},\qquad
s_0=\tau s_g\tau^{-1}=\begin{pmatrix}0&0&-p\\0&1_{2g-2}&0\\p^{-1}&0&0\end{pmatrix}.$$
We denote their images in $\tW$ by the same symbols.
Then $\bS=\{s_1,s_2,\ldots,s_g\}$ and $\tS=\bS\cup\{s_0\}$.
Moreover, $\tau s_i\tau^{-1}=s_{g-i}$ for $0\leq i\leq g$.

\begin{lemm}
\label{lemm:GSp-red-quot}
\label{reductive quotient symplectic}
Assume that $J$ is a maximal proper $\tau$-stable subset of $\tS$; equivalently, for some integer $0\leq i\leq \left\lfloor\frac{g}{2}\right\rfloor$, we have $J=\tS\setminus\{s_i,s_{g-i}\}$ if $2i<g$, and $J=\tS\setminus\{s_i\}$ if $2i=g$. As a connected reductive group over $\mathbb F_p$, $\overline P_J^1$ is isomorphic to
$$
\operatorname{Res}_{\mathbb F_{p^2}/\mathbb F_p}\Sp_{2i}\times\mathrm U_{g-2i},
$$
where $\mathrm U_n$ denotes the quasi-split unitary group associated with $\mathbb F_{p^2}/\mathbb F_p$.
\end{lemm}
The same statement can also be found in \cite[Lemma 3.5.2]{Harashita10}.
\begin{proof}
The absolute root datum obtained by deleting $\{s_i,s_{g-i}\}$ is that of $\Sp_{2i}\times\allowbreak\GL_{g-2i}\times\allowbreak\Sp_{2i}$.
The action of $\Ad(\tau)\circ\sigma$ on the cocharacter lattice $\mathbb Z^g$ is given by $(m_1,\ldots,m_g)\mapsto(-m_g,\ldots,-m_1)$.
The assertion follows immediately from this.
\end{proof}

If $g$ is odd (resp.\ even), then the unique very special subset of $\tS$ with respect to $\tau$ is $\tS\setminus\{s_0,s_g\}$ (resp.\ $\tS\setminus\{s_{\frac{g}{2}}\}$).
In these cases, the lemma gives $\overline P_J^1\cong\mathrm U_g$ (resp.\ $\overline P_J^1\cong\operatorname{Res}_{\mathbb F_{p^2}/\mathbb F_p}\Sp_g$).
If $p\geq5$, then $\J^1\cap P_J^+$ is torsion-free by Lemma \ref{lemm:torsion-free}, since it is contained in the corresponding subgroup for $\GL_{2g}$ (cf.\ the argument after Lemma \ref{lemm:GLveryspecial}), which is isomorphic to $V_{p,1}$.

The proof of Theorem \ref{main theo} below is similar to that of Theorem \ref{main theo GL}.
We omit some details that are identical to those in that proof.
For $s\in\{s_0,s_g\}$ and $x\in\overline{\mathbb F}_p$, set $u(x)=1_{2g}+p[x]E_{1,2g}$ if $s=s_0$ and $u(x)=1_{2g}+[x]E_{g+1,g}$ if $s=s_g$.

\begin{proof}[Proof of Theorem \ref{main theo}]
By Corollary \ref{projection universal homeomorphism} and the $\J^1$-equivariance of $\pi$, it suffices to prove the same assertion for $X_w(\tau)$.
Let $w_1,\ldots,w_r$ be as in Theorem \ref{simple}.
By Proposition \ref{spherical} and Theorem \ref{simple}, there is an irreducible component $C$ of $X_w(\tau)$ whose stabilizer in $\J^1$ is $\J^1\cap P_J$, where $J=\supp_\sigma(w_r)$.
Since $J$ is very special with respect to $\tau$, we have $J=\tS\setminus\{s_0,s_g\}$ if $g$ is odd and $J=\tS\setminus\{s_{\frac{g}{2}}\}$ if $g$ is even.

Note that $\J$ acts transitively on the irreducible components of $X_w(\tau)$ and normalizes $\J^1$.
By Lemma \ref{finite quotient}, it suffices to show that if $j\in\J^1\cap P_J$ has finite order and fixes $C$ pointwise, then $j=\pm1$.
Let $j$ be such an element.
By Proposition \ref{spherical} and Theorem \ref{simple}, the element $j$ fixes $Y(w_r)$ pointwise, where $Y(w_r)=\{xI\in P_J/I\mid x^{-1}\tau\sigma(x)\in Iw_rI\}$.
Let $\overline j$ be the image of $j$ under the isomorphism
$$(\J^1\cap P_J)/(\J^1\cap P_J^+)\xrightarrow{\sim}(\overline P_J^1)^{\Ad(\tau)\circ\sigma}.$$

Assume that $g$ is even.
Then $J=\tS\setminus\{s_{\frac{g}{2}}\}$ and $(\overline P_J^1)^{\Ad(\tau)\circ\sigma}\cong\Sp_g(\mathbb F_{p^2})$.
By Proposition \ref{center}, we have $\overline j=\pm1$.
Hence $j=\pm u^+$ for some $u^+\in\J^1\cap P_J^+$.
Since $u^+$ has finite order and $\J^1\cap P_J^+$ is torsion-free, we have $u^+=1$ and $j=\pm1$.

Assume that $g$ is odd.
Then $J=\tS\setminus\{s_0,s_g\}$ and $(\overline P_J^1)^{\Ad(\tau)\circ\sigma}\cong\mathrm U_g(\mathbb F_p)$.
By Proposition \ref{center}, we have $\overline j=c$ for some $c\in\mathbb F_{p^2}^\times$ satisfying $c^{p+1}=1$.
Set $z=\diag([c],\ldots,[c],[c^p],\ldots,[c^p])$, where each of $[c]$ and $[c^p]$ occurs $g$ times.
Then $z\in\J^1\cap P_J$ and its image in $(\overline P_J^1)^{\Ad(\tau)\circ\sigma}$ is $c$.
Hence $j=zu^+$ for some $u^+\in\J^1\cap P_J^+$.

Let $s\in\tS$ be the element corresponding to $k=r-1$ in Theorem \ref{simple} (i), so that $w_r=sw_{r-1}s$ and $\ell(w_r)=\ell(w_{r-1})-2$.
By Theorem \ref{simple} (ii), we have $s=s_0$ or $s=s_g$.
Let $f\colon X_{w_{r-1}}(\tau)\rightarrow X_{w_r}(\tau)$ be the morphism induced by Proposition \ref{DL method prop}.
For $yI\in Y(w_r)$, we have
$$f^{-1}(yI)=\{yu(x)sI\mid x\in\overline{\mathbb F}_p\}.$$
By assumption and Theorem \ref{simple}, $j$ fixes $f^{-1}(yI)$ pointwise.
Since $P_J^+$ is normal in $P_J$ and the image of $z$ in $P_J/P_J^+$ is central, both $y^{-1}u^+y$ and $z^{-1}y^{-1}zy$ lie in $P_J^+\cap\Sp_{2g}(\bQp)$.
Hence
$$y^{-1}u^+y=\begin{pmatrix}1_g+pA&pB\\ C&1_g+pD\end{pmatrix},\qquad z^{-1}y^{-1}zy=\begin{pmatrix}1_g+pA'&pB'\\ C'&1_g+pD'\end{pmatrix}$$
for some $A,B,C,D,A',B',C',D'\in\operatorname{Mat}_g(\breve{\Z}_p)$.
Then
$$jyu(x)sI=
\begin{cases}
yu(c^{1-p}(x+\overline{b_{1,g}}+\overline{b'_{1,g}}))sI & \text{if }s=s_0,\\
yu(c^{p-1}(x+\overline{c_{1,g}}+\overline{c'_{1,g}}))sI & \text{if }s=s_g.
\end{cases}$$
Here $\overline{b_{1,g}}$, $\overline{c_{1,g}}$, $\overline{b'_{1,g}}$ and $\overline{c'_{1,g}}$ are the reductions modulo $p$ of the $(1,g)$-entries of $B$, $C$, $B'$ and $C'$, respectively, and are independent of $x$.
Since $j$ fixes $f^{-1}(yI)$ pointwise, we have $c^{p-1}=1$.
Together with $c^{p+1}=1$, this gives $c^2=1$, and hence $c=\pm1$.
Thus $z=\pm1$, and $u^+=\pm j$ has finite order.
Since $\J^1\cap P_J^+$ is torsion-free, we have $u^+=1$ and $j=z=\pm1$.
\end{proof}

\begin{lemm}
\label{maximal support}
Let $w\in\SW\cap W_a\tau$. 
If $g$ is even (resp.\ odd) and $\supp_\sigma(w)$ is a maximal proper $\tau$-stable subset of $\tS$, then $\supp_\sigma(w)=\tS\setminus\{s_{\frac{g}{2}}\}$ (resp.\ $\supp_\sigma(w)=\tS\setminus\{s_{\frac{g-1}{2}},s_{\frac{g+1}{2}}\}$).
\end{lemm}
\begin{proof}
Write $w=v\tau$ with $v\in W_a$. Since $\supp_\sigma(w)$ is maximal, we have $\supp_\sigma(w)=\tS\setminus\{s_i,s_{g-i}\}$ for some $0\leq i\leq\left\lfloor\frac{g}{2}\right\rfloor$. Suppose that $i<\left\lfloor\frac{g}{2}\right\rfloor$. The simple reflections in different sets among $\{s_0,\ldots,s_{i-1}\}$, $\{s_{i+1},\ldots,s_{g-i-1}\}$, and $\{s_{g-i+1},\ldots,s_g\}$ commute. Since $w\in\SW$, it follows that $\supp(v)\subseteq\{s_0,\ldots,s_{i-1}\}$. Since $\tau$ exchanges the first and third sets, we have $\supp_\sigma(w)\subseteq\{s_0,\ldots,s_{i-1}\}\cup\{s_{g-i+1},\ldots,s_g\}$. This is impossible because $s_{\left\lfloor\frac{g}{2}\right\rfloor}\in\supp_\sigma(w)$. Hence $i=\left\lfloor\frac{g}{2}\right\rfloor$, as desired.
\end{proof}

The same argument as in the proof of Theorem \ref{main theo} proves the following result.
\begin{theo}
\label{maximal supersingular EO}
Assume that $w\in\SAdm(\mu)$, $\supp_\sigma(w)$ is a maximal proper $\tau$-stable subset of $\tS$, $g$ is even and $p\geq5$.
There exists an open dense subscheme $U\subseteq\pi(X_w(\tau))$ such that, for every closed point $xK\in U$, if $j\in\J^1$ has finite order and $jxK=xK$, then $j=\pm1$.
\end{theo}
\begin{proof}
Set $J=\supp_\sigma(w)$. By Lemma \ref{maximal support}, we have $J=\tS\setminus\{s_{\frac{g}{2}}\}$. By Proposition \ref{spherical} and the description following it, the irreducible components of $\pi(X_w(\tau))$ are the $\J$-translates of a fine Deligne--Lusztig variety $\pi(Y(w))$, whose stabilizer in $\J^1$ is $\J^1\cap P_J$.
Note that $\J$ normalizes $\J^1$. By Lemma \ref{finite quotient}, it suffices to show that if $j\in\J^1\cap P_J$ has finite order and fixes $\pi(Y(w))$ pointwise, then $j=\pm1$. Let $j$ be such an element, and let $\overline j$ be its image under the isomorphism
$$(\J^1\cap P_J)/(\J^1\cap P_J^+)\xrightarrow{\sim}(\overline P_J^1)^{\Ad(\tau)\circ\sigma}\cong\Sp_g(\mathbb F_{p^2})$$
given by Lemma \ref{lemm:GSp-red-quot}.
By Proposition \ref{center}, we have $\overline j=\pm1$. Hence $j=\pm u^+$ for some $u^+\in\J^1\cap P_J^+$. Since $u^+$ has finite order and $\J^1\cap P_J^+$ is torsion-free, we have $u^+=1$ and $j=\pm1$.
\end{proof}

If $g$ is odd and $J=\tS\setminus\{s_{\frac{g-1}{2}},s_{\frac{g+1}{2}}\}$, then $\overline P_J^1\cong\operatorname{Res}_{\mathbb F_{p^2}/\mathbb F_p}\Sp_{g-1}\times\mathrm U_1$ and $(\overline P_J^1)^{\Ad(\tau)\circ\sigma}\cong\Sp_{g-1}(\mathbb F_{p^2})\times\mathrm U_1(\mathbb F_p)$, where $\mathrm U_1(\mathbb F_p)=\{c\in\mathbb F_{p^2}^\times\mid c^{p+1}=1\}$. Its center $\{\pm1\}\times\mathrm U_1(\mathbb F_p)$ fixes $\pi(Y(w))$ pointwise, so the same argument does not apply when $g$ is odd.
In fact, there is a counterexample for every odd $g\geq3$ (see \cite[Remark 6.18]{KY24}).

\begin{exam}
\label{maximal supersingular EO example}
Assume that $g$ is even. The maximal supersingular Ekedahl--Oort stratum corresponds to $w=s_0(s_1s_0)(s_2s_1s_0)\cdots(s_{\frac{g}{2}-1}\cdots s_1s_0)\tau$. We have $\supp_\sigma(w)=\tS\setminus\{s_{\frac{g}{2}}\}$. Hence Theorem \ref{maximal supersingular EO} applies to the maximal supersingular Ekedahl--Oort stratum and generalizes \cite[Theorem A]{KY24}. Among the elements $w\in\SAdm(\mu)$ such that $\supp_\sigma(w)$ is a maximal proper $\tau$-stable subset of $\tS$, the unique element of minimal length is $w=s_0s_1\cdots s_{\frac{g}{2}-1}\tau$. The corresponding $\pi(X_w(\tau))$ has dimension $\ell(w)=\frac{g}{2}$.
\end{exam}

\begin{exam}
As observed in Remark \ref{dense Coxeter stratum}, there exists a unique $w\in\SAdm(\mu)$ of positive Coxeter type such that $\dim\pi(X_w(\tau))=\dim X_\mu(\tau)=\left\lfloor\frac{g^2}{4}\right\rfloor$.
There are other Ekedahl--Oort strata attached to elements $w$ of positive Coxeter type with $\supp_\sigma(w)=\tS$. For $g=7$, among the $44$ Ekedahl--Oort strata meeting the supersingular locus, $8$ are of this form. For $g=8$, among the corresponding $89$ strata, $12$ are of this form.
Determining the elements of minimal length among those of positive Coxeter type with $\supp_\sigma(w)=\tS$ is a nontrivial but tractable combinatorial problem. Although we do not give a proof here, computer experiments suggest that the minimal length is $\frac{g+3}{2}$ when $g$ is odd, $g+2$ when $g\equiv0\pmod 4$, and $g+3$ when $g\equiv2\pmod 4$. The corresponding varieties $\pi(X_w(\tau))$ have dimension $\frac{g+1}{2}$, $\frac{3g}{4}+1$, and $\frac{3(g+2)}{4}$, respectively (cf.\ \cite[Theorem 5.7(b)]{SSY23} for the dimension of $X_w(b)$ when $w$ is of positive Coxeter type). The experiments further suggest that there is one element of minimal length in each of the first two cases and two in the last case.

\end{exam}

\subsection{Further results in the case of $\operatorname{GSp}_{2g}$}
We use the notation in \S3.2.
The following proposition is a key ingredient in the proof of Theorem B.
\begin{prop}
\label{fixed locus positive Coxeter}
Assume that $w\in\SAdm(\mu)$ has positive Coxeter part, $\supp_\sigma(w)=\tS$, $g\geq4$ is even and $p\geq5$. Let $j\in\J^1\setminus\{\pm1\}$ be an element of finite order. Then the locus in $\pi(X_w(\tau))$ fixed pointwise by $j$ has codimension at least $2$.
\end{prop}
\begin{proof}
 By Corollary \ref{projection universal homeomorphism} and the $\mathbb J^1$-equivariance of $\pi$, it
suffices to prove the corresponding codimension estimate for the
fixed-point locus in $X_w(\tau)$. By Proposition~\ref{spherical} and Theorem~\ref{simple},
the irreducible components of $X_w(\tau)$ are pairwise disjoint. Hence, it is enough to prove the estimate on every
$j$-stable irreducible component $C$ of $X_w(\tau)$.
After conjugating $C$ and $j$ by an element
of $\mathbb J$, we may assume that the component of $X_{w_r}(\tau)$
associated with $C$ by the reduction sequence in Theorem~\ref{simple} is
$Y:=Y(w_r).$
Set $J=\operatorname{supp}_\sigma(w_r)$. Then
$j\in\mathbb J^1\cap P_J.$
Let
$
h\in\operatorname{Sp}_g(\mathbb F_{p^2})$
be the image of $j$ under
\[
(\mathbb J^1\cap P_J)/(\mathbb J^1\cap P_J^+)
\cong\operatorname{Sp}_g(\mathbb F_{p^2}).
\]
As in the proof of Theorem~\ref{main theo}, the torsion-freeness of
$\mathbb J^1\cap P_J^+$ shows that $h$ is noncentral.


Set $V=\overline{\mathbb F}_p^{\,g}$, equip $V$ with the symplectic form $\langle x,y\rangle=\sum_{i=1}^m(x_i y_{g+1-i}-x_{g+1-i}y_i)$, and let $\delta\colon V\longrightarrow V$ be given by $\delta(x_1,\ldots,x_g)=(x_1^{p^2},\ldots,x_g^{p^2})$. 
Write $g=2m$.
By \cite[\S4]{Lusztig76Green}, $Y$ can be identified with
$$
\left\{[x]\in\mathbb P(V)\,\middle|\,\langle x,\delta^i(x)\rangle=0\text{ for }1\leq i<m,\ \langle x,\delta^m(x)\rangle\neq0\right\}.
$$

Let $[x]\in Y$ be fixed by $h$, and write $hx=cx$ for some $c\in\overline{\mathbb F}_p^\times$. Since $h$ commutes with $\delta$, we have
$$
h\delta^i(x)=c^{p^{2i}}\delta^i(x)\qquad(0\leq i<g).
$$
The vectors $x,\delta(x),\ldots,\delta^{g-1}(x)$ form a basis of $V$ by the proof of \cite[Proposition 26 (i)]{Lusztig76Green}. Thus, $h$ has a basis of eigenvectors and is semisimple. In particular, every point of $Y$ fixed by $h$ belongs to $\bigcup_W\mathbb P(W)\subseteq \mathbb P(V)$, where $W$ runs over the eigenspaces of $h$.
Let $d$ be the smallest positive integer such that $c^{p^{2d}}=c$. Since $\delta$ permutes the eigenspaces, we have $d\mid g$, and each eigenspace has dimension $\frac{g}{d}$. Moreover, since $h\in\Sp_g(\mathbb F_{p^2})$ and $\langle x,\delta^m(x)\rangle\neq0$, we have $c^{1+p^{2m}}=1$.

If $d=1$, then $c=\pm1$ and hence $h=\pm1$, a contradiction. If $d\geq3$, every fixed point lies in the projective space associated with an eigenspace of $h$, which has dimension $\frac{g}{d}-1\leq m-2$.
Suppose that $d=2$. If $m$ is even, then $c^{p^{2m}}=c$, so $c^{1+p^{2m}}=1$ gives $c^2=1$, contradicting $d=2$. If $m$ is odd, then $m\geq3$ and $c^{p^2}=c^{-1}$. If $x$ and $x'$ lie in the same eigenspace, then $\langle x,x'\rangle=\langle hx,hx'\rangle=c^{\pm2}\langle x,x'\rangle$ and hence $\langle x,x'\rangle=0$. Since $\langle\ ,\ \rangle$ is non-degenerate on $V$ and $\delta$ exchanges the two eigenspaces, there exist $x,x'$ in the same eigenspace such that $\langle x,\delta(x')\rangle\neq0$. Since $\delta(ax')=a^{p^2}\delta(x')$, we have
$$
\langle x+ax',\delta(x+ax')\rangle=\langle x,\delta(x)\rangle+a^{p^2}\langle x,\delta(x')\rangle+a\langle x',\delta(x)\rangle+a^{p^2+1}\langle x',\delta(x')\rangle.
$$
The coefficient of $a^{p^2}$ is $\langle x,\delta(x')\rangle\neq0$, so this expression is nonzero for all but finitely many $a\in\overline{\mathbb F}_p$. Thus, for either eigenspace $W$, the equation $\langle x,\delta(x)\rangle=0$ is a nontrivial homogeneous equation on $\mathbb P(W)$ and defines a hypersurface of dimension $m-2$. Since every point of $Y$ satisfies this equation, $Y\cap\mathbb P(W)$ is contained in this hypersurface. It follows that the points of $Y$ fixed by $h$ form a locus of dimension at most $m-2$. This completes the proof.
\end{proof}

Let $v$ be as in Remark \ref{dense Coxeter stratum}.
For this $v$, the above conclusion also holds for odd $g$.
\begin{prop}
\label{fixed locus positive Coxeter odd}
Assume that $g\geq5$ is odd and $p\geq5$. Let $j\in\J^1\setminus\{\pm1\}$ be an element of finite order. Then the locus in $\pi(X_{\vp^\mu v}(\tau))$ fixed pointwise by $j$ has codimension at least $2$.
\end{prop}

To prove this proposition, we need to analyze several $\mathbb A^1$-fibrations in Theorem \ref{simple}. For this purpose, we need the following lemma.
\begin{lemm}
\label{rank one action}
Let $c\in\mathbb F_{p^2}^\times$ satisfy $c^{p+1}=1$, and set $z=\diag([c],\ldots,[c],[c^p],\ldots,[c^p])$, where each of $[c]$ and $[c^p]$ occurs $g$ times.
Let $j=zu^+$ for some $u^+\in\J^1\cap P_J^+$, where $J=\tS\setminus\{s_0,s_g\}$, and let $yI\in G(\bQp)/I$ be a point fixed by $j$.
Let $s\in\{s_0,s_g\}$.
Then there exist $a\in\overline{\mathbb F}_p$ and $\epsilon\in\{\pm1\}$ such that
$$
jyu(x)sI=yu(c^{2\epsilon}x+a)sI.
$$
\end{lemm}
Note that $y$ need not belong to $P_J$. This lemma can also be used to prove Theorem \ref{main theo}, but we present the two arguments separately for clarity.

\begin{proof}
Since $yI$ is fixed by $j$, we have $y^{-1}jy\in I\cap\Sp_{2g}(\bQp)$. The reduction modulo $p$ of $j$ is of the form
$$
\begin{pmatrix}c\cdot1_g&0\\ *&c^p\cdot1_g\end{pmatrix}.
$$
Since $y^{-1}jy$ is conjugate to $j$, the characteristic polynomial of its reduction modulo $p$ is $(T-c)^g(T-c^p)^g$. Write the reduction of $y^{-1}jy$ as $(c_{ij})$. Then $c_{ii}\in\{c,c^p\}$ for every $i$. Since this matrix is lower triangular and symplectic, we have $c_{ii}c_{2g+1-i,2g+1-i}=1$. A direct matrix calculation gives, for some $a\in\overline{\mathbb F}_p$,
$$
jyu(x)sI=
\begin{cases}
yu(c_{11}^2(x+a))sI & \text{if }s=s_0,\\
yu(c_{gg}^{-2}(x+a))sI & \text{if }s=s_g.
\end{cases}
$$
Since $c_{11},c_{gg}\in\{c,c^p\}=\{c,c^{-1}\}$, for some $a\in\overline{\mathbb F}_p$ and $\epsilon\in\{\pm1\}$, we have
$$
jyu(x)sI=yu(c^{2\epsilon}x+a)sI.
$$
This finishes the proof.
\end{proof}

\begin{proof}[Proof of Proposition \ref{fixed locus positive Coxeter odd}]
By Corollary \ref{projection universal homeomorphism}, it suffices to prove the assertion for an irreducible component $C$ of $X_{\vp^\mu v}(\tau)$ stabilized by $j$.
By Proposition \ref{spherical} and Theorem \ref{simple}, the image $\overline j$ of $j$ in $\mathrm U_g(\mathbb F_p)$ stabilizes a Deligne--Lusztig variety $Y$ of Coxeter type for $\mathrm U_g$ over $\mathbb F_p$.
Writing $g=2m+1$, we first claim that, for every non-central element $h\in\mathrm U_g(\mathbb F_p)$, the locus in $Y$ fixed pointwise by $h$ has dimension at most $m-2<\dim Y=m$.

Set $V=\overline{\mathbb F}_p^{\,g}$, equip $V$ with the Hermitian form $\langle x,y\rangle=\sum_{i=1}^g x_i y_i^p$, and let $\delta\colon V\longrightarrow V$ be given by $\delta(x_1,\ldots,x_g)=(x_1^{p^2},\ldots,x_g^{p^2})$. By \cite[\S4]{Lusztig76Green}, $Y$ can be identified with
$$
\left\{[x]\in\mathbb P(V)\,\middle|\,\langle x,\delta^i(x)\rangle=0\text{ for }0\leq i<m,\ \langle x,\delta^m(x)\rangle\neq0\right\}.
$$

Let $[x]\in Y$ be fixed by $h$, and write $hx=cx$ for some $c\in\overline{\mathbb F}_p^\times$. Since $h$ commutes with $\delta$, we have
$$
h\delta^i(x)=c^{p^{2i}}\delta^i(x)\qquad(0\leq i<g).
$$
The vectors $x,\delta(x),\ldots,\delta^{g-1}(x)$ form a basis of $V$ by the proof of \cite[Proposition 26 (i)]{Lusztig76Green}. Thus, $h$ has a basis of eigenvectors and is semisimple. In particular, every point of $Y$ fixed by $h$ belongs to $\bigcup_W\mathbb P(W)\subseteq\mathbb P(V)$, where $W$ runs over the eigenspaces of $h$.
Let $d$ be the smallest positive integer such that $c^{p^{2d}}=c$. Since $\delta$ permutes the eigenspaces, we have $d\mid g$, and each eigenspace has dimension $\frac{g}{d}$. Moreover, since $h\in\mathrm U_g(\mathbb F_p)$ and $\langle x,\delta^m(x)\rangle\neq0$, we have $c^{1+p^{2m+1}}=1$.

If $d=1$, then $h=cI$ is central, a contradiction. Hence $d\geq3$, since $g$ is odd. Every fixed point lies in the projective space associated with an eigenspace of $h$, which has dimension $\frac{g}{d}-1\leq m-2$.
This proves the claim.

Set $J=\tS\setminus\{s_0,s_g\}$. It remains to consider the case
$$
j=zu^+,\qquad z=\diag(\underbrace{[c],\ldots,[c]}_g,\underbrace{[c^p],\ldots,[c^p]}_g),\qquad u^+\in\J^1\cap P_J^+,
$$
where $c\in\mathbb F_{p^2}^\times$ satisfies $c^{p+1}=1$ and $c\neq\pm1$.
Note that we have $v=s_gs_{g-1}\cdots s_1$ and
$$
\vp^\mu v=s_0(s_1s_0)(s_2s_1s_0)\cdots(s_{g-2}\cdots s_0)\tau.
$$

For $G=\GSp_{2g}$, we have the following relations in $W_a$:
$$
s_0s_1s_0s_1=s_1s_0s_1s_0,\qquad s_gs_{g-1}s_gs_{g-1}=s_{g-1}s_gs_{g-1}s_g,
$$
$$
s_0s_i=s_is_0\quad(2\leq i\leq g),\qquad s_gs_i=s_is_g\quad(0\leq i\leq g-2).
$$
Thus, for each element $w\in W_a$, the total number of occurrences of $s_0$ and $s_g$ is independent of the choice of its reduced expression.
We denote this number by $o(w)$.
Let $w'=s_iw\tau s_i\tau^{-1}$ for some $0\leq i\leq g$.
By symmetry, we may assume that $w\neq w'$ and $\ell(s_iw)<\ell(w)$.
If $\ell(w)=\ell(w')$, then both $w\tau s_i\tau^{-1}$ and $s_iw'$ are reduced expressions, and hence $o(w)=o(w')$.
If $\ell(w')=\ell(w)-2$, then $s_iw'\tau s_i\tau^{-1}$ is a reduced expression of $w$. Hence $o(w')=o(w)$ if $1\leq i\leq g-1$, and $o(w')=o(w)-2$ if $i=0$ or $g$.

Set $w_0=\vp^\mu v$, and let $w_1,\ldots,w_r$ be as in Theorem \ref{simple}.
For $u\in W_a$, we extend the definition of $o$ to $W_a\tau$ by setting $o(u\tau):=o(u)$.
Then $o(w_0)=g-1=2m$ and $o(w_r)=0$.
It follows from the above argument that the set
$$
\left\{0\leq k<r\,\middle|\,\ell(w_{k+1})=\ell(w_k)-2,\ w_{k+1}=s_iw_ks_i\text{ for some }i\in\{0,g\}\right\}
$$
has cardinality $m$.
For $0\leq k<r$, let $f_k\colon X_{w_k}(\tau)\longrightarrow X_{w_{k+1}}(\tau)$ be the map provided by Proposition \ref{DL method prop}. Since $f_k$ is $\J$-equivariant, we have
$$
X_{w_k}(\tau)^j\subseteq f_k^{-1}\bigl(X_{w_{k+1}}(\tau)^j\bigr),
$$
where $X_w(\tau)^j$ denotes the locus in $X_w(\tau)$ fixed by $j$.
Moreover, if $k$ belongs to the above set, it follows from Lemma \ref{rank one action} and $c\neq\pm1$ that for every $y\in X_{w_{k+1}}(\tau)^j$, the fixed locus in the fiber $f_k^{-1}(y)$ consists of a single point.
Thus, in this case, the restriction of $f_k$ to $X_{w_k}(\tau)^j$ is (universally) bijective, and hence
$$
\dim X_{w_k}(\tau)^j=\dim X_{w_{k+1}}(\tau)^j.
$$
Since the above set has cardinality $m\ge 2$, we obtain
$$
\dim X_{\vp^\mu v}(\tau)^j\leq\dim X_{\vp^\mu v}(\tau)-m\leq\dim X_{\vp^\mu v}(\tau)-2.
$$
This completes the proof.
\end{proof}

\subsection{The case of $\operatorname{GSO}_{2g}$}
Let $G=\operatorname{GSO}_{2g}$ and assume that $g\geq4$ is even.
The main goal of this subsection is to prove the following theorem.
\begin{theo}
\label{main theo GSO}
Assume that $w\in\SAdm(\mu)$ has positive Coxeter part, $\supp_\sigma(w)=\tS$ and $p\geq5$.
There exists an open dense subscheme $U\subseteq\pi(X_w(\tau))$ such that, for every closed point $xK\in U$, if $j\in\J^1$ has finite order and $jxK=xK$, then $j=\pm1$.
\end{theo}

For a $\Q_p$-algebra $R$, define
$$
J_\tau^1(R)=\left\{j\in\operatorname{SO}_{2g}(R\otimes_{\Q_p}\bQp)\,\middle|\,j^{-1}\tau\sigma(j)=\tau\right\},
$$
where $\sigma$ acts trivially on $R$.
This defines a reductive group $J_\tau^1$ over $\Q_p$, which is an inner form of $\operatorname{SO}_{2g}$, and we set $\J^1=J_\tau^1(\Q_p)$.
For a $\tau$-stable subset $J\subsetneq\tS$, we define $\overline P_J^1$ as in \S3.2, with $\Sp_{2g}$ replaced by $\operatorname{SO}_{2g}$.

For $1\leq i<g$, let $s_i$ be the permutation matrix associated with $(i\ i+1)(2g-i\ 2g+1-i)$.
Set
$$
 s_g=\begin{pmatrix}1_{g-2}&0&0&0\\0&0&1_2&0\\0&1_2&0&0\\0&0&0&1_{g-2}\end{pmatrix},\qquad
 s_0=\tau s_g\tau^{-1}=\begin{pmatrix}0&0&p\cdot1_2\\0&1_{2g-4}&0\\p^{-1}\cdot1_2&0&0\end{pmatrix}.
$$
We denote their images in $\tW$ by the same symbols.
Then $\bS=\{s_1,s_2,\ldots,s_g\}$ and $\tS=\bS\cup\{s_0\}$.
Moreover, $\tau s_i\tau^{-1}=s_{g-i}$ for $0\leq i\leq g$.
Set
$$
\omega=\begin{pmatrix}
0&0&0&0&0&p\\
0&1_{g-2}&0&0&0&0\\
0&0&0&1&0&0\\
0&0&1&0&0&0\\
0&0&0&0&1_{g-2}&0\\
p^{-1}&0&0&0&0&0
\end{pmatrix}\in\operatorname{SO}_{2g}(\bQp).
$$
A direct calculation shows that $\omega\tau=\tau\omega$ and that conjugation by $\omega$ exchanges $s_0$ with $s_1$ and $s_{g-1}$ with $s_g$ and fixes all the other elements of $\tS$.
In particular, $\omega$ normalizes $I$ and represents an element of $\Omega$.

\begin{lemm}
\label{reductive quotient GSO}
Assume that $J$ is a very special subset of $\tS$ with respect to $\tau$; equivalently, we have $J=\tS\setminus\{s_0,s_g\}$ or $J=\tS\setminus\{s_1,s_{g-1}\}$.
These two subsets are conjugate by $\omega$.
As a connected reductive group over $\mathbb F_p$, $\overline P_J^1$ is isomorphic to $\mathrm U_g.$
\end{lemm}
\begin{proof}
For either subset in the statement, the absolute root datum is that of $\GL_g$, and $\Ad(\tau)\circ\sigma$ induces its nontrivial diagram automorphism and acts on its center by $c\mapsto c^{-p}$.
Hence $\overline P_J^1\cong\mathrm U_g$.
\end{proof}

Moreover, if $p\geq5$, then $\J^1\cap P_J^+$ is torsion-free, since it is contained in the corresponding subgroup for $\GL_{2g}$, which is isomorphic to $V_{p,1}$.

The proof of Theorem \ref{main theo GSO} below is similar to that of Theorem \ref{main theo GL}.
We omit some details that are identical to those in that proof.
For $s\in\{s_0,s_g\}$ and $x\in\overline{\mathbb F}_p$, set $u(x)=1_{2g}+p[x](E_{1,2g-1}-E_{2,2g})$ if $s=s_0$ and $u(x)=1_{2g}+[x](E_{g+1,g-1}-E_{g+2,g})$ if $s=s_g$.

\begin{proof}[Proof of Theorem \ref{main theo GSO}]
By Corollary \ref{projection universal homeomorphism} and the $\J^1$-equivariance of $\pi$, it suffices to prove the same assertion for $X_w(\tau)$.
Let $w_1,\ldots,w_r$ be as in Theorem \ref{simple}.
By Proposition \ref{spherical} and Theorem \ref{simple}, there is an irreducible component $C$ of $X_w(\tau)$ whose stabilizer in $\J^1$ is $\J^1\cap P_J$, where $J=\tS\setminus\{s_0,s_g\}$ or $J=\tS\setminus\{s_1,s_{g-1}\}$.
Since the two subsets are conjugate by $\omega$, we may assume that $J=\tS\setminus\{s_0,s_g\}$.

Note that $\J$ acts transitively on the irreducible components of $X_w(\tau)$ and normalizes $\J^1$.
By Lemma \ref{finite quotient}, it suffices to show that if $j\in\J^1\cap P_J$ has finite order and fixes $C$ pointwise, then $j=\pm1$.
Let $j$ be such an element.
By Proposition \ref{spherical} and Theorem \ref{simple}, the element $j$ fixes $Y(w_r)$ pointwise, where $Y(w_r)=\{xI\in P_J/I\mid x^{-1}\tau\sigma(x)\in Iw_rI\}$.
Let $\overline j$ be the image of $j$ under the isomorphism
$$
(\J^1\cap P_J)/(\J^1\cap P_J^+)\xrightarrow{\sim}(\overline P_J^1)^{\Ad(\tau)\circ\sigma}\cong\mathrm U_g(\mathbb F_p).
$$
By Proposition \ref{center}, we have $\overline j=c$ for some $c\in\mathbb F_{p^2}^\times$ satisfying $c^{p+1}=1$.
Set $z=\diag([c],\ldots,[c],[c^p],\ldots,[c^p])$, where each of $[c]$ and $[c^p]$ occurs $g$ times.
Then $j=zu^+$ for some $u^+\in\J^1\cap P_J^+$.

Let $s\in\tS$ be the element corresponding to $k=r-1$ in Theorem \ref{simple} (i), so that $w_r=sw_{r-1}s$ and $\ell(w_r)=\ell(w_{r-1})-2$.
By Theorem \ref{simple} (ii), we have $s=s_0$ or $s=s_g$.
Let $f\colon X_{w_{r-1}}(\tau)\rightarrow X_{w_r}(\tau)$ be the morphism induced by Proposition \ref{DL method prop}.
For $yI\in Y(w_r)$, we have
$$
f^{-1}(yI)=\{yu(x)sI\mid x\in\overline{\mathbb F}_p\}.
$$
By assumption and Theorem \ref{simple}, $j$ fixes $f^{-1}(yI)$ pointwise.
Since $P_J^+$ is normal in $P_J$ and the image of $z$ in $P_J/P_J^+$ is central, both $y^{-1}u^+y$ and $z^{-1}y^{-1}zy$ lie in $P_J^+\cap\operatorname{SO}_{2g}(\bQp)$.
Hence
$$y^{-1}u^+y=\begin{pmatrix}1_g+pA&pB\\ C&1_g+pD\end{pmatrix},\qquad z^{-1}y^{-1}zy=\begin{pmatrix}1_g+pA'&pB'\\ C'&1_g+pD'\end{pmatrix}$$
for some $A,B,C,D,A',B',C',D'\in\operatorname{Mat}_g(\breve{\Z}_p)$.
Then
$$jyu(x)sI=
\begin{cases}
yu(c^{1-p}(x+\overline{b_{1,g-1}}+\overline{b'_{1,g-1}}))sI & \text{if }s=s_0,\\
yu(c^{p-1}(x+\overline{c_{1,g-1}}+\overline{c'_{1,g-1}}))sI & \text{if }s=s_g.
\end{cases}$$
Here $\overline{b_{1,g-1}}$, $\overline{c_{1,g-1}}$, $\overline{b'_{1,g-1}}$ and $\overline{c'_{1,g-1}}$ are the reductions modulo $p$ of the $(1,g-1)$-entries of $B$, $C$, $B'$ and $C'$, respectively, and are independent of $x$.
Since $j$ fixes $f^{-1}(yI)$ pointwise, the equations for $x=0$ and $x=1$ imply that $c^{p-1}=1$.
Together with $c^{p+1}=1$, this gives $c^2=1$, and hence $c=\pm1$.
Thus $z=\pm1$, and $u^+=z^{-1}j$ has finite order.
Since $\J^1\cap P_J^+$ is torsion-free, we have $u^+=1$ and $j=z=\pm1$.
\end{proof}

\section{The locus with $a$-number at least $2$}
Let $G=\GSp_{2g}$.
Throughout this section, assume that $g\geq 4$.
We use the notation in \S3.2.
We identify $X_\mu(\tau)$ with the set of Dieudonn\'e lattices $M\subseteq\bQp^{2g}$ satisfying $M\supseteq\tau\sigma(M)\supseteq pM$ and $M^\vee=dM$ for some $d\in\bQp^\times$.
For a $p$-divisible group $X$ over $\overline{\F}_p$, define its $a$-number by $a(X)=\dim_{\overline{\F}_p}\operatorname{Hom}_{\overline{\F}_p}(\alpha_p,X)$.
If $M$ is the Dieudonn\'e module of $X$, then $a(X)=a(M)\coloneqq\dim_{\overline{\F}_p}M/(\tau\sigma(M)+p(\tau\sigma)^{-1}(M))$.

Let $v$ be as in Remark \ref{dense Coxeter stratum}.
In this section, we study $X_\mu(\tau)'\coloneqq X_\mu(\tau)\setminus\pi(X_{\vp^\mu v}(\tau))$.
The $a=1$ locus in $X_\mu(\tau)$ coincides with $\pi(X_{\vp^\mu v}(\tau))$. Moreover, it is the $\J$-orbit of a top-dimensional $\J$-stratum (cf.\ \cite[Proposition 5.11]{CV18}). Thus, $X_\mu(\tau)'$ is the $a\geq2$ locus and is also a union of $\J$-strata.

\subsection{Irreducible components of $X_\mu(\tau)'$}

By Remark \ref{dense Coxeter stratum}, we already know that $X_\mu(\tau)'$ is equidimensional of codimension $1$ in $X_\mu(\tau)$. Set $\widetilde\nu_\tau=(0,1,0,1,\ldots,0,1)\in X_*(T)$, where $(0,1)$ occurs $g$ times. By Theorem \ref{flat bijection}, $\widetilde\nu_\tau$ corresponds to the unique $\J$-orbit of top-dimensional $\J$-strata in $X_\mu(\tau)$. 
The elements $\nu\in W_0\mu$ satisfying $\nu\leq\nu_\tau$ and $\langle\rho,\widetilde\nu_\tau-\nu\rangle=1$ are precisely $s_{2k}\widetilde\nu_\tau$ for $1\leq k\leq\left\lfloor\frac{g}{2}\right\rfloor$. Indeed, $\widetilde\nu_\tau-\nu$ is a non-negative integral sum of simple coroots, so the second condition implies that it is a simple coroot. Since $\nu\in W_0\mu$, it follows that $\nu=s_{2k}\widetilde\nu_\tau$ for some $1\leq k\leq\left\lfloor\frac{g}{2}\right\rfloor$.
Since $X_\mu(\tau)'$ is the complement of the unique top-dimensional
$\mathbb J$-orbit, these codimension-one $\mathbb J$-strata are contained
in $X_\mu(\tau)'$. We construct the corresponding small cocharacters
$\lambda^{(k)}$ and show that the irreducible components of
$X_\mu(\tau)'$ are precisely the $\mathbb J$-translates of
$\overline{X_\mu^{\lambda^{(k)}}(\tau)}$.

For $0\leq i\leq g$, let $\alpha_i\in\Pi$ be such that $s_{\widetilde{\alpha}_i}=s_i$. For $\lambda\in X_*(T)$, write $\lambda(i)$ for its $i$-th coordinate and set $\lambda_i=\lambda_{\alpha_i}$. Explicitly, $\lambda_0=\lambda(1)-\lambda(2g)-1$ and $\lambda_i=\lambda(i+1)-\lambda(i)$ for $1\leq i\leq g$. Then $\lambda\in\cAm$ is small if and only if $\min\{\lambda_i,\lambda_{g-i}\}\leq0$ for every $0\leq i\leq g$.

Assume that $g=2m$.
For $1\leq k<m$, define $\lambda^{(k)}$, up to a central cocharacter, by $\lambda_0^{(k)}=0$, $\lambda_{2m}^{(k)}=1$, $\lambda_m^{(k)}=1-m$, and
$$
(\lambda_i^{(k)},\lambda_{2m-i}^{(k)})=
\begin{cases}
(0,0)&\text{if $i=m-k$},\\
(1,0)&\text{if $i<m-k$ and $i$ is odd, or if $i>m-k$ and $i$ is even},\\
(0,1)&\text{if $i<m-k$ and $i$ is even, or if $i>m-k$ and $i$ is odd}
\end{cases}
$$
for $1\leq i<m$.
Define $\lambda^{(m)}$, up to a central cocharacter, by $\lambda_0^{(m)}=-m$, $\lambda_{2m}^{(m)}=1-m$, $\lambda_m^{(m)}=0$, and
$$
(\lambda_i^{(m)},\lambda_{2m-i}^{(m)})=
\begin{cases}
(1,0)&\text{if $i$ is odd},\\
(0,1)&\text{if $i$ is even}
\end{cases}
$$
for $1\leq i<m$.

Assume that $g=2m+1$.
For $1\leq k<m$, define $\lambda^{(k)}$, up to a central cocharacter, by $\lambda_0^{(k)}=-m$, $\lambda_{2m+1}^{(k)}=1-m$, and
$$
(\lambda_i^{(k)},\lambda_{2m+1-i}^{(k)})=
\begin{cases}
(0,0)&\text{if $i=k$},\\
(1,0)&\text{if $i<k$ and $i$ is odd, or if $i>k$ and $i$ is even},\\
(0,1)&\text{if $i<k$ and $i$ is even, or if $i>k$ and $i$ is odd}
\end{cases}
$$
for $1\leq i\leq m$.
Define $\lambda^{(m)}$, up to a central cocharacter, by $\lambda_0^{(m)}=0$, $\lambda_{2m+1}^{(m)}=1$, $\lambda_m^{(m)}=-\left\lfloor\frac{m}{2}\right\rfloor$, $\lambda_{m+1}^{(m)}=-\left\lceil\frac{m}{2}\right\rceil$, and
$$
(\lambda_i^{(m)},\lambda_{2m+1-i}^{(m)})=
\begin{cases}
(1,0)&\text{if $i$ is odd},\\
(0,1)&\text{if $i$ is even}
\end{cases}
$$
for $1\leq i<m$.

Choosing each $\lambda^{(k)}$ so that $\lambda^{(k)}(1)=0$, we can describe them as follows.
If $g=2m$, then $\lambda^{(k)}$ for $1\leq k<m$ and $\lambda^{(m)}$ are given by
$$
\lambda^{(k)}(i)=
\begin{cases}
\left\lfloor\frac{i}{2}\right\rfloor&\text{if $1\leq i\leq m-k$},\\
\left\lfloor\frac{i-1}{2}\right\rfloor&\text{if $m-k<i\leq m$},\\
-1-\left\lfloor\frac{g-i-1}{2}\right\rfloor&\text{if $m<i\leq m+k$},\\
-1-\left\lfloor\frac{g-i}{2}\right\rfloor&\text{if $m+k<i\leq g$},
\end{cases}
\quad
\lambda^{(m)}(i)=
\begin{cases}
\left\lfloor\frac{i}{2}\right\rfloor&\text{if $1\leq i\leq m$},\\
\left\lfloor\frac{i-1}{2}\right\rfloor&\text{if $m<i\leq g$}
\end{cases}
$$
with $\lambda^{(k)}(2g+1-i)=-1-\lambda^{(k)}(i)$ for $1\leq k<m$ and $\lambda^{(m)}(2g+1-i)=m-1-\lambda^{(m)}(i)$.
If $g=2m+1$, then $\lambda^{(k)}$ for $1\leq k<m$ and $\lambda^{(m)}$ are given by
$$
\lambda^{(k)}(i)=
\begin{cases}
\left\lfloor\frac{i}{2}\right\rfloor&\text{if $1\leq i\leq k$},\\
\left\lfloor\frac{i-1}{2}\right\rfloor&\text{if $k<i\leq g-k$},\\
\left\lfloor\frac{i-2}{2}\right\rfloor&\text{if $g-k<i\leq g$},
\end{cases}
\quad
\lambda^{(m)}(i)=
\begin{cases}
\left\lfloor\frac{i}{2}\right\rfloor&\text{if $1\leq i\leq m$},\\
0&\text{if $i=m+1$},\\
-\left\lfloor\frac{g-i+2}{2}\right\rfloor&\text{if $m+1<i\leq g$}
\end{cases}
$$
with $\lambda^{(k)}(2g+1-i)=m-1-\lambda^{(k)}(i)$ for $1\leq k<m$ and $\lambda^{(m)}(2g+1-i)=-1-\lambda^{(m)}(i)$.
\begin{lemm}
\label{lemmma:codimone}
The set $\{\lambda^{(k)}\mid1\leq k\leq\left\lfloor\frac{g}{2}\right\rfloor\}$ is a complete set of representatives for the $\Omega$-orbits of those $\lambda\in\cAm^{\mathrm{sm}}$ for which $X_\mu^\lambda(\tau)$ has codimension $1$ in $X_\mu(\tau)$.
\end{lemm}
\begin{proof}
For any $\lambda\in X_*(T)$, $\lambda^{\natural} = \tau(\lambda)-\lambda$ satisfies
$$
\lambda^\natural(1)=\frac{1-\lambda_0+\lambda_g}{2},\qquad
\lambda^\natural(i+1)-\lambda^\natural(i)=\lambda_{g-i}-\lambda_i\quad(1\leq i<g),
$$
and
$$
\lambda^\natural(i)=\lambda^\natural(g+1-i),\qquad
\lambda^\natural(g+i)=1-\lambda^\natural(i)\quad(1\leq i\leq g).
$$

Set $\lambda=\lambda^{(k)}$.
If $g=2m$ and $1\leq k<m$, then
$$
\lambda^\natural(i)=
\begin{cases}
1&\text{if $i\leq m-k$ and $i$ is odd, or if $i>m-k$ and $i$ is even},\\
0&\text{otherwise}
\end{cases}
$$
for $1\leq i\leq m$.
If $g=2m$ and $k=m$, then
$$
\lambda^\natural(i)=
\begin{cases}
1&\text{if $i$ is odd},\\
0&\text{if $i$ is even}
\end{cases}
$$
for $1\leq i\leq m$.
If $g=2m+1$ and $1\leq k<m$, then
$$
\lambda^\natural(i)=
\begin{cases}
1&\text{if $i\leq k$ and $i$ is odd, or if $i>k$ and $i$ is even},\\
0&\text{otherwise}
\end{cases}
$$
for $1\leq i\leq m+1$.
If $g=2m+1$ and $k=m$, then
$$
\lambda^\natural(i)=
\begin{cases}
1&\text{if $i\leq m$ and $i$ is odd},\\
0&\text{otherwise}
\end{cases}
$$
for $1\leq i\leq m+1$.
Hence $\lambda\in\cAm^{\mathrm{sm}}$.

By Theorem \ref{flat bijection}, it remains to show that $\lambda^\flat=s_{2k}\widetilde\nu_\tau$ for $1\leq k\leq\left\lfloor\frac{g}{2}\right\rfloor$.
We only consider the case where $g=2m$ and $1\leq k<m$, as the other cases are similar.
First note that $\lambda^\flat$ can be computed by viewing $\lambda$ as a cocharacter of $\GL_{2g}$ (cf.\ \cite[\S 6.2]{Shimada6}).
The element $\epsilon_\lambda$ is determined by requiring that $\epsilon_\lambda^{-1}\lambda$ be dominant and that coordinates of equal value be ordered by decreasing index.
We then obtain $\lambda^\flat$ by arranging the coordinates of $\lambda^\natural$ in this order.
We may assume that $\lambda(1)=0$. Then $\lambda$ has exactly $g$ nonnegative coordinates, all lying between $0$ and $\left\lfloor\frac{m-1}{2}\right\rfloor$.
Namely,
$$
\begin{aligned}
\lambda=(&\underbrace{\left\lfloor\tfrac{1}{2}\right\rfloor,\ldots,\left\lfloor\tfrac{m-k}{2}\right\rfloor}_{m-k},
\underbrace{\left\lfloor\tfrac{m-k}{2}\right\rfloor,\ldots,\left\lfloor\tfrac{m-1}{2}\right\rfloor}_{k},
\ast,\ldots,\ast,\\
&\underbrace{\left\lfloor\tfrac{0}{2}\right\rfloor,\ldots,\left\lfloor\tfrac{m-k-1}{2}\right\rfloor}_{m-k},
\underbrace{\left\lfloor\tfrac{m-k-1}{2}\right\rfloor,\ldots,\left\lfloor\tfrac{m-2}{2}\right\rfloor}_{k},
\ast,\ldots,\ast).
\end{aligned}
$$
If $m$ is even, then
$$
\begin{array}{c@{\quad}c}
\begin{aligned}
\lambda^\natural=(&
\underbrace{1,0,\ldots,1,0}_{m-k},
\underbrace{0,1,\ldots,0,1}_{k},
\ast,\ldots,\ast,\\
&\underbrace{0,1,\ldots,0,1}_{m-k},
\underbrace{1,0,\ldots,1,0}_{k},
\ast,\ldots,\ast)
\end{aligned}
&
\begin{aligned}
\lambda^\natural=(&
\underbrace{1,0,\ldots,0,1}_{m-k},
\underbrace{1,0,\ldots,1}_{k},
\ast,\ldots,\ast,\\
&\underbrace{0,1,\ldots,1,0}_{m-k},
\underbrace{0,1,\ldots,0}_{k},
\ast,\ldots,\ast)
\end{aligned}
\end{array}
$$
for even and odd $k$, respectively.
If $m$ is odd, then
$$
\begin{array}{c@{\quad}c}
\begin{aligned}
\lambda^\natural=(&
\underbrace{1,0,\ldots,0,1}_{m-k},
\underbrace{1,0,\ldots,1,0}_{k},
\ast,\ldots,\ast,\\
&\underbrace{0,1,\ldots,1,0}_{m-k},
\underbrace{0,1,\ldots,0,1}_{k},
\ast,\ldots,\ast)
\end{aligned}
&
\begin{aligned}
\lambda^\natural=(&
\underbrace{1,0,\ldots,1,0}_{m-k},
\underbrace{0,1,\ldots,0}_{k},
\ast,\ldots,\ast,\\
&\underbrace{0,1,\ldots,0,1}_{m-k},
\underbrace{1,0,\ldots,1}_{k},
\ast,\ldots,\ast)
\end{aligned}
\end{array}
$$
for even and odd $k$, respectively.
Note that the coordinates denoted by $\ast$ are determined by the displayed coordinates.
Thus,
$$
\lambda^\flat=(\underbrace{0,1,\ldots,0,1}_{2k-2},0,0,1,1,
\underbrace{0,1,\ldots,0,1}_{g-2k-2},\ast,\ldots,\ast).
$$
This finishes the proof.
\end{proof}

\begin{prop}
\label{parahoric stabilizers a-number}
For every $C\in\Irr(X_\mu(\tau)')$, its stabilizer $H$ in $\J$ is parahoric.
\end{prop}
\begin{proof}
Since $X_\mu(\tau)'$ is a union of Ekedahl--Oort strata, there are $w\in\SAdm(\mu)$ and $Z\in\Irr \bigl(X_w(\tau)\bigr)$ such that $C=\overline{\pi(Z)}$. By \cite[Proposition 3.1.4]{ZZ20}, the stabilizer $P$ of $Z$ in $\J$ is parahoric. Moreover, the stabilizer $H$ of $C$ in $\J$ is bounded (cf.\ \cite[Lemma 4.2.2]{ZZ20}), and $H$ is contained in ${\mathbb J^0} := \ker (\kappa|_{\mathbb J})$ since $H$ stabilizes the connected component of the affine Grassmannian containing $C$. Since $P \subset H$, it is parahoric.
\end{proof}
\begin{coro}
\label{irreducible components a-number}
Let $\lambda^{(k)}$ be as above.
If $g$ is even, then
$$
\Pi(\lambda^{(k)})=
\begin{cases}
\Pi\setminus\{\alpha_{\frac{g}{2}}\}&\text{if $1\leq k<\frac{g}{2}$},\\
\Pi\setminus\{\alpha_0,\alpha_g\}&\text{if $k=\frac{g}{2}$}.
\end{cases}
$$
If $g$ is odd, then
$$
\Pi(\lambda^{(k)})=
\begin{cases}
\Pi\setminus\{\alpha_0,\alpha_g\}&\text{if $1\leq k<\frac{g-1}{2}$},\\
\Pi\setminus\{\alpha_{\frac{g-1}{2}},\alpha_{\frac{g+1}{2}}\}&\text{if $k=\frac{g-1}{2}$}.
\end{cases}
$$
In either case, we have
$$\Irr\bigl(X_\mu(\tau)'\bigr)=\bigl\{j\overline{X_\mu^{\lambda^{(k)}}(\tau)}\mid 1\leq k\leq\left\lfloor\tfrac{g}{2}\right\rfloor,\ j\in\J/(\J\cap P_{\Pi(\lambda^{(k)})})\bigr\}.$$
\end{coro}
\begin{proof}
Since $p(\tau)$ exchanges $\alpha_i$ and $\alpha_{g-i}$, the descriptions of $\Pi(\lambda^{(k)})$ follow directly from the values of $\lambda_i^{(k)}$ given above.
For each $k$, $\overline{X_\mu^{\lambda^{(k)}}(\tau)}$ is an irreducible component of $X_\mu(\tau)'$.
Let $C\in\Irr(X_\mu(\tau)')$. By \cite[Theorem 2.10]{Gortz19}, only finitely many $\J$-strata meet $C$. We may therefore choose a $\J$-stratum $S$ such that $S\cap C$ is open and dense in $C$.
If $j$ belongs to the stabilizer of $C$ in $\J$, then $jS\cap C$ is again open and dense in $C$. This implies that $jS=S$. Thus the stabilizer of $S$ in $\J$ contains that of $C$. It is bounded and contains a parahoric subgroup by Proposition \ref{parahoric stabilizers a-number}, and hence is parahoric by the same argument as in Proposition \ref{parahoric stabilizers a-number}. By Theorem \ref{parahoric J-strata} and Lemma \ref{lemmma:codimone}, $S$ is irreducible and corresponds to a small cocharacter $\lambda^{(k)}$. Thus $C=\overline S=j\overline{X_\mu^{\lambda^{(k)}}(\tau)}$ for some $1\leq k\leq\left\lfloor\frac{g}{2}\right\rfloor$ and $j\in\J$.
\end{proof}

More generally, for every $1\leq r\leq g$, Harashita \cite{Harashita04} proved that the locus in $\mathscr S_g$ with $a$-number at least $r$ is equidimensional and gave a formula for the number of its irreducible components.

\subsection{The main theorem for $X_\mu(\tau)'$}
For $J\subsetneq\tS$ and $\lambda\in X_*(T)$, let $\overline Q_\lambda$ be the image of $P_J\cap p^\lambda Kp^{-\lambda}$ under the reduction map $P_J\longrightarrow\overline P_J$. It is a proper parabolic subgroup of $\overline P_J$ if and only if $\lambda_\alpha\neq0$ for some $s_{\ta}\in J$. 
The reduction map $P_J\rightarrow\overline P_J$ induces a projection
$$\rho_\lambda\colon P_Jp^\lambda K/K\cong P_J/(P_J\cap p^\lambda Kp^{-\lambda})\longrightarrow\overline P_J/\overline Q_\lambda,\qquad hp^\lambda K\longmapsto\overline{h}\,\overline Q_\lambda.$$
The fibers are isomorphic to affine spaces.
If $P_J=K$, this is the map in \cite[(1.4.4)]{Zhu17}.
Let $\overline Q_\lambda^1$ be the image of $P_J\cap p^\lambda Kp^{-\lambda}\cap\Sp_{2g}(\bQp)$ under the reduction map $P_J\cap\Sp_{2g}(\bQp)\longrightarrow\overline P_J^1$.
Then $\overline Q_\lambda^1=\overline P_J^1\cap\overline Q_\lambda$ and the natural inclusion induces an isomorphism
$\overline P_J^1/\overline Q_\lambda^1\cong\overline P_J/\overline Q_\lambda.$

The goal of this subsection is to prove the following theorem.
\begin{theo}
\label{main theorem a-number}
Assume that $g\geq4$ and $p\geq5$.
There exists an open dense subscheme $U\subset X_\mu(\tau)'$ such that, for every closed point $xK\in U$, if $j\in\J^1$ has finite order and $jxK=xK$, then $j=\pm1$.
\end{theo}
\begin{proof}
Since $X_\mu(\tau)'$ is locally perfectly of finite type, its irreducible components form a locally finite family.
Therefore, by Corollary \ref{irreducible components a-number}, it suffices to find such an open dense subscheme in each irreducible component $\overline{X_\mu^{\lambda^{(k)}}(\tau)}$ for $1\leq k\leq\left\lfloor\frac{g}{2}\right\rfloor$.
Set $\lambda=\lambda^{(k)}$ and $J=\Pi(\lambda^{(k)})$.

The stabilizer of $\overline{X_\mu^\lambda(\tau)}$ in $\J^1$ is $\J^1\cap P_J$. By Lemma \ref{finite quotient}, it suffices to show that if $j\in\J^1\cap P_J$ has finite order and fixes $\overline{X_\mu^\lambda(\tau)}$ pointwise, then $j=\pm1$.
Let $j$ be such an element.
Let $\overline j\in (\overline P_J^1)^{\Ad(\tau)\circ\sigma}$ be the image of $j$.
The projection $\rho_\lambda$ maps $(\J^1\cap P_J)p^\lambda K$ onto $(\overline P_J^1)^{\Ad(\tau)\circ\sigma}/((\overline P_J^1)^{\Ad(\tau)\circ\sigma}\cap\overline Q_\lambda^1)$.
Since $(\J^1\cap P_J)p^\lambda K\subseteq\overline{X_\mu^\lambda(\tau)}$, the element $\overline j$ fixes this quotient pointwise. Note that $\overline Q_\lambda^1$ is a proper parabolic subgroup of $\overline P_J^1$ and that $\tau$ acts transitively on the set of connected components of $J$. The proof of Proposition \ref{center} therefore shows that $\overline j$ lies in the center of $(\overline P_J^1)^{\Ad(\tau)\circ\sigma}$.

Assume that $g$ is even and $k<\frac{g}{2}$. Then $J=\tS\setminus\{s_{\frac{g}{2}}\}$ and $(\overline P_J^1)^{\Ad(\tau)\circ\sigma}\cong\Sp_g(\mathbb F_{p^2})$, so $\overline j=\pm1$ (cf.\ Lemma \ref{reductive quotient symplectic}). Hence $j=\pm u^+$ for some $u^+\in\J^1\cap P_J^+$. Since $p\geq5$, by the argument after Lemma \ref{lemm:GSp-red-quot}, the group $\J^1\cap P_J^+$ is torsion-free. As $u^+$ has finite order, we have $u^+=1$ and $j=\pm1$.

Assume that $g$ is even and $k=\frac{g}{2}$. Then $J=\tS\setminus\{s_0,s_g\}$ and $(\overline P_J^1)^{\Ad(\tau)\circ\sigma}\cong\mathrm U_g(\mathbb F_p)$. Hence $\overline j=c$ for some $c\in\mathbb F_{p^2}^\times$ satisfying $c^{p+1}=1$. Set $z=\diag([c],\ldots,[c],[c^p],\ldots,[c^p])$, where each entry occurs $g$ times. Then $j=zu^+$ for some $u^+\in\J^1\cap P_J^+$. For $x\in\overline{\mathbb F}_p$, set $u(x)=1_{2g}+p^{\frac{g}{2}-2}[x]E_{2g,1}$.
Then
$$(u(x)p^\lambda)^{-1}\tau\sigma(u(x)p^\lambda)=(1_{2g}+[x^p]E_{g,g+1})p^{\lambda^\natural}p(\tau)(1_{2g}-[x]E_{g,g+1})\in Kp^\mu K.$$
Thus $u(x)p^\lambda K\in X_\mu^\lambda(\tau)$ for every $x\in\overline{\mathbb F}_p$.
Since $j$ fixes $p^\lambda K$ and $zp^\lambda K=p^\lambda K$, we have $p^{-\lambda}u^+p^\lambda\in K$.
Since $j$ fixes $u(x)p^\lambda K$, we have $u^+u(x)p^\lambda K=u(c^{1-p}x)p^\lambda K$, and hence
$$p^{-\lambda}u(-c^{1-p}x)u^+u(x)p^\lambda\in K.$$
Since $u^+\in P_J^+$, the $(2g,1)$-entry of this matrix satisfies
$$\left(p^{-\lambda}u(-c^{1-p}x)u^+u(x)p^\lambda\right)_{2g,1}\in p^{-1}\bigl([x]-[c^{1-p}x]\bigr)+\breve{\Z}_p.$$
Thus $[x]-[c^{1-p}x]\in p\breve{\Z}_p$, so $x=c^{1-p}x$. Evaluating at $x=1$, we obtain $c^{p-1}=1$. Together with $c^{p+1}=1$, this gives $c=\pm1$. Thus $j=\pm u^+$, so $u^+$ has finite order. 
Although $J$ is not very special in this case, the corresponding parahoric type for $\operatorname{GL}_{2g}$ is very special. Hence, by the same argument after Lemma \ref{lemm:GSp-red-quot}, the group $\J^1\cap P_J^+$ is torsion-free.
Hence $u^+=1$ and $j=\pm1$.

Assume that $g$ is odd and $k<\frac{g-1}{2}$. Then $J=\tS\setminus\{s_0,s_g\}$ and $(\overline P_J^1)^{\Ad(\tau)\circ\sigma}\cong\mathrm U_g(\mathbb F_p)$. For $x\in\overline{\mathbb F}_p$, set $u(x)=1_{2g}+p^{\frac{g-5}{2}}[x]E_{2g,1}$. The same calculation as above shows that $u(x)p^\lambda K\in X_\mu^\lambda(\tau)$ for every $x\in\overline{\mathbb F}_p$. The remaining parts of the argument are also the same as above.

Assume that $g=2m+1$ and $k=m$. Then $J=\tS\setminus\{s_m,s_{m+1}\}$ and $(\overline P_J^1)^{\Ad(\tau)\circ\sigma}\cong\Sp_{2m}(\mathbb F_{p^2})\times\mathrm U_1(\mathbb F_p)$. Hence $\overline j=(\epsilon,c)$ for some $\epsilon\in\{\pm1\}$ and $c\in\mathbb F_{p^2}^\times$ satisfying $c^{p+1}=1$. Set
$$z=\diag(\underbrace{\epsilon,\ldots,\epsilon}_{m},[c],\underbrace{\epsilon,\ldots,\epsilon}_{2m},[c^p],\underbrace{\epsilon,\ldots,\epsilon}_{m}).$$
Then $j=zu^+$ for some $u^+\in\J^1\cap P_J^+$. For $x,y\in\overline{\mathbb F}_p$ satisfying $y^p+y+x^{p+1}=0$, set $$u(x,y)=1_{2g}+[x](E_{m+1,g}-E_{g+1,3m+2})+[y](E_{1,g}-E_{g+1,2g}).$$ A direct calculation gives
\begin{align*}
&(u(x,y)p^\lambda)^{-1}\tau\sigma(u(x,y)p^\lambda)\\
&=\bigl(1_{2g}+[x^p](E_{3m+2,2g}-E_{1,m+1})\bigr)p^{\lambda^\natural}p(\tau)
\bigl(1_{2g}-[x](E_{3m+2,2g}-E_{1,m+1})\bigr)\\
&\quad\cdot\bigl(1_{2g}+p^{-1}([y]+[y^p]+[x^{p+1}])(E_{1,g}-E_{g+1,2g})\bigr)\in Kp^\mu K.
\end{align*}
Thus $u(x,y)p^\lambda K\in X_\mu^\lambda(\tau)$. Moreover, $z^{-1}u(x,y)z=u(\epsilon c^px,y)$. Since $j$ fixes $u(x,y)p^\lambda K$, we have $u^+u(x,y)p^\lambda K=u(\epsilon c^px,y)p^\lambda K$, and hence
$$p^{-\lambda}u(\epsilon c^px,y)^{-1}u^+u(x,y)p^\lambda\in K.$$
Set $h=p^{-\lambda}u^+p^\lambda\in K$. Since $u^+\in P_J^+\subset I$, we have $h_{m+1,m+1},h_{g,g}\in1+p\breve{\Z}_p$, $h_{m+1,1}\in\breve{\Z}_p$, and $h_{g,m+1},h_{g,1}\in p\breve{\Z}_p$. The $(m+1,g)$-entry of the above matrix is
\begin{align*}
&h_{m+1,g}+p^{-1}[x]h_{m+1,m+1}+p^{-1}[y]h_{m+1,1}-p^{-1}[\epsilon c^px]h_{g,g}\\
&\qquad-p^{-2}[\epsilon c^px][x]h_{g,m+1}-p^{-2}[\epsilon c^px][y]h_{g,1}.
\end{align*}
Multiplying this entry by $p$ and reducing modulo $p$, we obtain
$$
(1-\epsilon c^p)x+\overline h_{m+1,1}y-\epsilon c^p\,\overline{p^{-1}h_{g,m+1}}x^2-\epsilon c^p\,\overline{p^{-1}h_{g,1}}xy=0
$$
for every $(x,y)$ satisfying $y^p+y+x^{p+1}=0$, where the bars denote reduction modulo $p$. The elements $x,y,x^2,xy$ are linearly independent in $\overline{\mathbb F}_p[x,y]/(y^p+y+x^{p+1})$. Hence $\epsilon c^p=1$, and therefore $c=\epsilon$. Thus $j=\pm u^+$.
If $u^+$ lies in a torsion-free subgroup, then we have $u^+=1$ and $j=\pm 1$, since $u^+$ has finite order.

Thus, it remains to show that $u^+$ lies in a torsion-free subgroup.
Let $\omega$ be as in the proof of Theorem \ref{main theo GL}. Since $V_{p,1}$ is torsion-free by $p\geq5$ and Lemma \ref{lemm:torsion-free}, it suffices to show that $\omega^{-m}u^+\omega^m\in V_{p,1}$.
Since $u^+\in\J^1\cap P_J^+$, we have
$$
\overline{\omega^{-m}u^+\omega^m}=
\begin{pmatrix}
1&0&0&0\\
\mathbf u&1_{g-1}&0&0\\
*&*&1&0\\
*&*&\sigma(\mathbf u)&1_{g-1}
\end{pmatrix}\in\GL_{2g}(\overline{\mathbb F}_p),
\qquad
\mathbf u=\begin{pmatrix}
\overline{u^+_{m+2,m+1}}\\
\vdots\\
\overline{u^+_{3m+1,m+1}}
\end{pmatrix}.
$$
Thus, it suffices to show that $\mathbf u=0$.
For $m+2\leq i\leq g$, we have $s_is_{g-i}\in\J^1\cap P_J$. Let $r\in\J^1\cap P_J$ be any product of these elements. Since $u^+$ fixes $\overline{X_\mu^\lambda(\tau)}$ pointwise and $P_J^+$ is normal in $P_J$, the element $r^{-1}u^+r$ belongs to $\J^1\cap P_J^+$ and
$$
r^{-1}u^+r\,u(x,y)p^\lambda K=u(x,y)p^\lambda K.
$$
Set $h'=p^{-\lambda}r^{-1}u^+rp^\lambda$. Computing the $(m+1,g)$-entry as above, multiplying it by $p$, and reducing modulo $p$, we obtain
$$
\overline{h'_{m+1,1}}y-\overline{p^{-1}h'_{g,m+1}}x^2-\overline{p^{-1}h'_{g,1}}xy=0
$$
for every $(x,y)$ satisfying $y^p+y+x^{p+1}=0$. Hence
$$
\overline{(r^{-1}u^+r)_{g,m+1}}=\overline{p^{-1}h'_{g,m+1}}=0.
$$
Since $r$ is arbitrary, we obtain $\mathbf u=0$, as desired.
\end{proof}

For comparison, let $\lambda=(1^{(g-1)},0,1,0^{(g-1)})$. Then $\lambda$ is a small cocharacter corresponding to a one-dimensional parahoric $\J$-stratum, and $\Pi(\lambda)=\tS\setminus\{s_1,s_{g-1}\}$. The connected components of $\Pi(\lambda)$ are $\{s_0\}$, $\{s_2,\ldots,s_{g-2}\}$ and $\{s_g\}$. If $g\geq4$, $\tau$ does not act transitively on the set of connected components of $\Pi(\lambda)$. Thus, the argument in the proof of Proposition \ref{center} does not apply in this case.
In fact, by Lemma \ref{reductive quotient symplectic}, we have $\overline P_{\Pi(\lambda)}^1\cong\operatorname{Res}_{\mathbb F_{p^2}/\mathbb F_p}\Sp_2\times\mathrm U_{g-2}$. On the other hand, after base change to $\overline{\mathbb F}_p$, the parabolic subgroup $\overline Q_\lambda^1$ is isomorphic to $\Sp_2\times\GL_{g-2}\times B_{\Sp_2}$, where $B_{\Sp_2}$ is a Borel subgroup of $\Sp_2$. Consequently, $\overline P_{\Pi(\lambda)}^1/\overline Q_\lambda^1\cong\Sp_2/B_{\Sp_2}$, and $\mathrm U_{g-2}(\mathbb F_p)$ acts trivially on this quotient. Note that in this case, the projection $\rho_\lambda$ is an isomorphism.

\begin{rema}
\label{Oort conjecture from a-number loci}
Since $X_\mu(\tau)$ is equidimensional, every irreducible component of $X_\mu(\tau)'$ is contained in an irreducible component of $X_\mu(\tau)$. Since $\J$ acts transitively on the irreducible components of $X_\mu(\tau)$ by Remark \ref{dense Coxeter stratum}, every irreducible component of $X_\mu(\tau)$ contains an irreducible component of $X_\mu(\tau)'$. Hence, if $j\in\J^1$ has finite order and fixes an irreducible component of $X_\mu(\tau)$ pointwise, then $j=\pm1$ by Theorem \ref{main theorem a-number}. Therefore, for $g\geq4$ and $p\geq5$, Oort's conjecture follows from Lemma \ref{finite quotient}. Similarly, by the same argument with $\pi(X_w(\tau))$ in place of $X_\mu(\tau)'$, Theorem \ref{main theo} and Theorem \ref{maximal supersingular EO} imply Oort's conjecture for $g\geq3$ and even $g$, respectively, when $p\geq5$.
\end{rema}

\begin{rema}
Analogues of Theorem \ref{main theorem a-number} should also hold for $\GL_{2g}$ and $\operatorname{GSO}_{4m}$. Their proofs should be less involved, since the possible parahoric stabilizers of irreducible components are more restricted in these cases. However, a difficulty in the latter case is that the analogue of Theorem \ref{flat bijection} has not been established.
\end{rema}

\begin{coro}
\label{codimension automorphism locus}
Assume that $g\geq4$ and $p\geq5$. Let $U\subset X_\mu(\tau)$ be the locus whose closed points are those $xK$ such that $j=\pm1$ whenever $j\in\J^1$ has finite order and $jxK=xK$.
Then $U$ is open, and $X_\mu(\tau)\setminus U$ has codimension at least $2$ in $X_\mu(\tau)$.
\end{coro}
\begin{proof}
Let $Z\subset X_\mu(\tau)$ be affine open. Since $Z$ is quasi-compact, $Z$ is contained in $\bigcup_{i=1}^rK\vp^{\lambda_i}K/K$ for some $\lambda_1,\ldots,\lambda_r\in X_*(T)_+$. Every $j\in\J^1$ fixing a point of $Z$ lies in the compact subset $\bigcup_{i=1}^r(\J^1\cap K\vp^{\lambda_i}K\vp^{-\lambda_i}K)$ of $\J^1$. Thus $Z\setminus U$ is a finite union of closed subsets of the form $\{xK\in Z\mid jxK=xK\}$ (cf.\ the proof of Lemma \ref{finite quotient}). Therefore $U\cap Z$ is open, and hence $U$ is open.
Let $v$ be as in Remark \ref{dense Coxeter stratum}. By Lemma \ref{finite quotient}, Proposition \ref{fixed locus positive Coxeter} and Proposition \ref{fixed locus positive Coxeter odd}, $\pi(X_{\vp^\mu v}(\tau))\setminus U$ has codimension at least $2$ in $X_\mu(\tau)$. By Theorem \ref{main theorem a-number}, $X_\mu(\tau)'\setminus U$ has positive codimension in $X_\mu(\tau)'$. Since $X_\mu(\tau)'$ has codimension $1$ in $X_\mu(\tau)$, the assertion follows.
\end{proof}

\bibliographystyle{myamsplain}
\bibliography{reference}

\providecommand{\bysame}{\leavevmode\hbox to3em{\hrulefill}\thinspace}
\providecommand{\MR}{\relax\ifhmode\unskip\space\fi MR }
\providecommand{\MRhref}[2]{%
  \href{http://www.ams.org/mathscinet-getitem?mr=#1}{#2}
}
\providecommand{\href}[2]{#2}
\begin{thebibliography}{10}

\bibitem{BS17}
B.~Bhatt and P.~Scholze, \emph{Projectivity of the {W}itt vector affine
  {G}rassmannian}, Invent. Math. \textbf{209} (2017), no.~2, 329--423.
  \MR{3674218}

\bibitem{BR06}
C.~Bonnaf\'e and R.~Rouquier, \emph{On the irreducibility of
  {D}eligne-{L}usztig varieties}, C. R. Math. Acad. Sci. Paris \textbf{343}
  (2006), no.~1, 37--39. \MR{2241956}

\bibitem{CO11}
C.-L. Chai and F.~Oort, \emph{Monodromy and irreducibility of leaves}, Ann. of
  Math. (2) \textbf{173} (2011), no.~3, 1359--1396.

\bibitem{CV18}
M.~Chen and E.~Viehmann, \emph{Affine {D}eligne-{L}usztig varieties and the
  action of {$J$}}, J. Algebraic Geom. \textbf{27} (2018), no.~2, 273--304.
  \MR{3764277}

\bibitem{Dragutinovic24}
D.~Dragutinovi\'c, \emph{{O}ort's conjecture and automorphisms of supersingular
  curves of genus four}, arXiv:2405.01282 (2024).

\bibitem{EMO01}
S.~J. Edixhoven, B.~J.~J. Moonen, and F.~Oort, \emph{Open problems in algebraic
  geometry}, Bull. Sci. Math. \textbf{125} (2001), no.~1, 1--22.

\bibitem{Gashi10}
Q.~R. Gashi, \emph{On a conjecture of {K}ottwitz and {R}apoport}, Ann. Sci.
  \'{E}c. Norm. Sup\'{e}r. (4) \textbf{43} (2010), no.~6, 1017--1038.
  \MR{2778454}

\bibitem{Gortz19}
U.~G\"{o}rtz, \emph{Stratifications of affine {D}eligne-{L}usztig varieties},
  Trans. Amer. Math. Soc. \textbf{372} (2019), no.~7, 4675--4699. \MR{4009395}

\bibitem{GH10}
U.~G\"{o}rtz and X.~He, \emph{Dimensions of affine {D}eligne-{L}usztig
  varieties in affine flag varieties}, Doc. Math. \textbf{15} (2010),
  1009--1028. \MR{2745691}

\bibitem{GH15}
\bysame, \emph{Basic loci of {C}oxeter type in {S}himura varieties}, Camb. J.
  Math. \textbf{3} (2015), no.~3, 323--353. \MR{3393024}

\bibitem{GHN19}
U.~G\"{o}rtz, X.~He, and S.~Nie, \emph{Fully {H}odge-{N}ewton decomposable
  {S}himura varieties}, Peking Math. J. \textbf{2} (2019), no.~2, 99--154.
  \MR{4060001}

\bibitem{GHN24}
U.~G\"ortz, X.~He, and S.~Nie, \emph{Basic loci of {C}oxeter type with
  arbitrary parahoric level}, Canad. J. Math. \textbf{76} (2024), no.~1,
  126--172. \MR{4687768}

\bibitem{GS26}
U.~G\"{o}rtz and S.~Schr\"{o}er, \emph{Deligne--{L}usztig varieties whose
  canonical divisors have negativity}, arXiv:2605.02522 (2026).

\bibitem{HV18}
P.~Hamacher and E.~Viehmann, \emph{Irreducible components of minuscule affine
  {D}eligne-{L}usztig varieties}, Algebra Number Theory \textbf{12} (2018),
  no.~7, 1611--1634. \MR{3871504}

\bibitem{Harashita04}
S.~Harashita, \emph{The {$a$}-number stratification on the moduli space of
  supersingular abelian varieties}, J. Pure Appl. Algebra \textbf{193} (2004),
  no.~1--3, 163--191.

\bibitem{Harashita10}
\bysame, \emph{Ekedahl--{O}ort strata contained in the supersingular locus and
  {D}eligne--{L}usztig varieties}, J. Algebraic Geom. \textbf{19} (2010),
  no.~3, 419--438.

\bibitem{HV12}
U.~Hartl and E.~Viehmann, \emph{Foliations in deformation spaces of local
  {$G$}-shtukas}, Adv. Math. \textbf{229} (2012), no.~1, 54--78. \MR{2854170}

\bibitem{He09}
X.~He, \emph{{$G$}-stable pieces and partial flag varieties}, Representation
  theory, Contemp. Math., vol. 478, Amer. Math. Soc., Providence, RI, 2009,
  pp.~61--70. \MR{2513266}

\bibitem{He23}
\bysame, \emph{A generalization of cyclic shift classes}, Selecta Math. (N.S.)
  \textbf{29} (2023), no.~5, Paper No. 83.

\bibitem{HN14}
X.~He and S.~Nie, \emph{Minimal length elements of extended affine {W}eyl
  groups}, Compos. Math. \textbf{150} (2014), no.~11, 1903--1927. \MR{3279261}

\bibitem{HNY22}
X.~He, S.~Nie, and Q.~Yu, \emph{Affine {D}eligne--{L}usztig varieties with
  finite {C}oxeter parts}, Algebra Number Theory \textbf{18} (2024), no.~9,
  1681--1714.

\bibitem{HZZ21}
X.~He, R.~Zhou, and Y.~Zhu, \emph{Stabilizers of irreducible components of
  affine {D}eligne--{L}usztig varieties}, J. Eur. Math. Soc. (JEMS) \textbf{27}
  (2025), no.~6, 2387--2441.

\bibitem{Ibukiyama20}
T.~Ibukiyama, \emph{Principal polarizations of supersingular abelian surfaces},
  J. Math. Soc. Japan \textbf{72} (2020), no.~4, 1161--1180.

\bibitem{KP19}
V.~Karemaker and R.~Pries, \emph{Fully maximal and fully minimal abelian
  varieties}, J. Pure Appl. Algebra \textbf{223} (2019), no.~7, 3031--3056.

\bibitem{KYY21}
V.~Karemaker, F.~Yobuko, and C.-F. Yu, \emph{Mass formula and {O}ort's
  conjecture for supersingular abelian threefolds}, Adv. Math. \textbf{386}
  (2021), Paper No. 107812.

\bibitem{KY24}
V.~Karemaker and C.-F. Yu, \emph{Supersingular {E}kedahl-{O}ort strata and
  {O}ort's conjecture}, arXiv:2406.19748 (2024).

\bibitem{KY26}
\bysame, \emph{{O}ort's conjecture on supersingular abelian varieties in odd
  characteristic}, arXiv:2608.16405 (2026).

\bibitem{KR00}
R.~Kottwitz and M.~Rapoport, \emph{Minuscule alcoves for {${\rm GL}_n$} and
  {$G{\rm Sp}_{2n}$}}, Manuscripta Math. \textbf{102} (2000), no.~4, 403--428.
  \MR{1785323}

\bibitem{Kottwitz85}
R.~E. Kottwitz, \emph{Isocrystals with additional structure}, Compositio Math.
  \textbf{56} (1985), no.~2, 201--220. \MR{809866}

\bibitem{Lusztig76Green}
G.~Lusztig, \emph{On the {G}reen polynomials of classical groups}, Proc. London
  Math. Soc. (3) \textbf{33} (1976), no.~3, 443--475.

\bibitem{Lusztig76}
\bysame, \emph{{C}oxeter orbits and eigenspaces of {F}robenius}, Invent. Math.
  \textbf{38} (1976/77), no.~2, 101--159. \MR{453885}

\bibitem{MT11}
G.~Malle and D.~Testerman, \emph{Linear algebraic groups and finite groups of
  {L}ie type}, Cambridge Studies in Advanced Mathematics, vol. 133, Cambridge
  University Press, Cambridge, 2011.

\bibitem{Nie22}
S.~Nie, \emph{Irreducible components of affine {D}eligne-{L}usztig varieties},
  Camb. J. Math. \textbf{10} (2022), no.~2, 433--510.

\bibitem{NSY25}
S.~Nie, F.~Schremmer, and Q.~Yu, \emph{Lifting {D}eligne-{L}usztig reduction
  and geometric {C}oxeter type elements}, arXiv:2507.18453 (2025).

\bibitem{Oort77}
F.~Oort, \emph{Singularities of coarse moduli schemes}, S\'{e}minaire
  d'Alg\`{e}bre Paul Dubreil, 29\`{e}me ann\'{e}e (Paris, 1975--1976), Lecture
  Notes in Mathematics, vol. 586, Springer, Berlin, 1977, pp.~61--76.

\bibitem{PR08}
G.~Pappas and M.~Rapoport, \emph{Twisted loop groups and their affine flag
  varieties}, Adv. Math. \textbf{219} (2008), no.~1, 118--198, With an appendix
  by T. Haines and M. Rapoport. \MR{2435422}

\bibitem{RZ96}
M.~Rapoport and T.~Zink, \emph{Period spaces for {$p$}-divisible groups},
  Annals of Mathematics Studies, vol. 141, Princeton University Press,
  Princeton, NJ, 1996. \MR{1393439}

\bibitem{Schremmer23}
F.~Schremmer, \emph{Newton strata in {L}evi subgroups}, Manuscripta Math.
  \textbf{175} (2024), no.~1-2, 513--519. \MR{4790570}

\bibitem{SSY23}
F.~Schremmer, R.~Shimada, and Q.~Yu, \emph{Affine {D}eligne-{L}usztig varieties
  of positive {C}oxeter type}, arXiv:2312.02630 (2023).

\bibitem{Shimada4}
R.~Shimada, \emph{The {E}kedahl-{O}ort stratification and the semi-module
  stratification}, Canad. J. Math. \textbf{78} (2026), no.~3, 695--732.

\bibitem{Shimada6}
\bysame, \emph{On {$\mathbb J$}-strata with parahoric stabilizers in affine
  {D}eligne-{L}usztig varieties}, arXiv:2606.03062 (2026).

\bibitem{STinprep}
R.~Shimada and T.~Takamatsu, \emph{The dimension of the moduli space of
  principally polarized abelian varieties with extra automorphisms}, In
  preparation.

\bibitem{ST24}
\bysame, \emph{On the supersingular locus of the {S}iegel modular variety of
  genus $3$ or $4$}, arXiv:2403.19505, to appear in Ann. Inst. Fourier
  (Grenoble).

\bibitem{Takaya25}
Y.~Takaya, \emph{Equidimensionality of affine {D}eligne-{L}usztig varieties in
  mixed characteristic}, Adv. Math. \textbf{465} (2025), Paper No. 110153, 22.
  \MR{4864699}

\bibitem{Viehmann08b}
E.~Viehmann, \emph{The global structure of moduli spaces of polarized
  {$p$}-divisible groups}, Doc. Math. \textbf{13} (2008), 825--852.

\bibitem{Viehmann26}
\bysame, \emph{Oort's conjecture on automorphisms of generic supersingular
  abelian varieties}, arXiv:2603.06033 (2026).

\bibitem{Wang21}
H.~Wang, \emph{Deligne-{L}usztig varieties and basic {EKOR} strata}, Canad.
  Math. Bull. \textbf{64} (2021), no.~2, 349--367. \MR{4273205}

\bibitem{ZZ20}
R.~Zhou and Y.~Zhu, \emph{Twisted orbital integrals and irreducible components
  of affine {D}eligne-{L}usztig varieties}, Camb. J. Math. \textbf{8} (2020),
  no.~1, 149--241. \MR{4085434}

\bibitem{Zhu17}
X.~Zhu, \emph{Affine {G}rassmannians and the geometric {S}atake in mixed
  characteristic}, Ann. of Math. (2) \textbf{185} (2017), no.~2, 403--492.
  \MR{3612002}

\end{thebibliography}

\end{document}